\documentclass{aimsessay}
\usepackage{booktabs} % \toprule, \midrule, \bottomrule
\usepackage{array}    % \newcolumntype

\usepackage[below]{placeins} % use \FloatBarrier in the body
\usepackage{sidecap}
\usepackage{graphicx}
\usepackage{moreverb}   % \verbatimtabinput
\usepackage{hyperref}
\usepackage{fancyhdr} %For header image on title page

\usepackage[latin1]{inputenc}
\usepackage{wrapfig}
\usepackage[numbers]{natbib} 
\swapnumbers 
\theoremstyle{plain}
\newtheorem{thm}[subsection]{Theorem}
\theoremstyle{definition} 
\newtheorem{lem}[subsection]{Lemma}
\newtheorem{cor}[subsection]{Corollary}

\newtheorem{pro}[subsection]{Proposition}
\newtheorem{exa}[subsection]{Example}
\newtheorem{defn}[subsection]{Definition}
\newtheorem{rem}[subsection]{Remark}
\numberwithin{equation}{section}
\title{\vspace*{-4cm} The Fundamental Group of Torus Knots}
\author{Ilyas Aderogba Mustapha (mustapha.ilyas@aims-cameroon.org)\\
African Institute for Mathematical Sciences (AIMS)\\
Cameroon\\\\
{\small Supervised by: Dr. Paul Arnaud Songhafouo Tsopm\'en\'e and Prof. Donald Stanley,}\\
{\small University of Regina, Canada}%
}
\date{{\small 18 May 2018}\\%
  {\scriptsize\it Submitted in Partial Fulfillment of 
    a Structured Masters Degree at AIMS-Cameroon}
}
\begin{document}
\pagestyle{empty}
%The pagestyle 'empty' is used by \maketitle for the 'report' document class. 
% In order to include a header image on the first page, the pagestyle 'empty' 
% must be redefined. However, the redefinition should only be local which 
% is why we group the commands below within braces
{ 
\setlength{\headheight}{95pt} %Ensure that the header image fits within the header.
% Without this line, fancyhdr will increase \headheight to fit the header image. 
% In addition, the increased \headheight will be used throughout the document.
% There will thus be a large empty space at the top of every page...
\fancypagestyle{empty}{%
  \fancyhf{}% Clear header/footer
  \renewcommand{\headrulewidth}{0pt} %Remove the horizontal line
  \fancyhead[C]{\includegraphics[width = \textwidth]{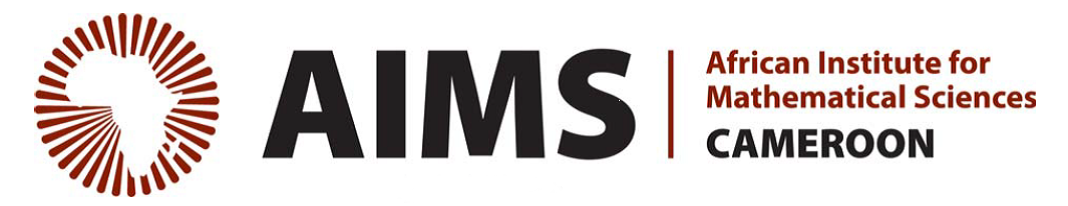}}% Logo
} %end \fancypagestyle

\maketitle
}

% All other files are included via \input. 
% To compile in texmaker while viewing any of those
% without having to switch back, use
%   Options > Define Current Document as 'Master Document'
%To compile in TeXstudio while viewing any of those
% without having to switch back, use
%   Options >Root Document> Detect Authomatically
% To not have to close a PDF, remove viewpdf from quickbuild 
% and open the PDF (once) manually: it will auto-refresh or with control-r
% 
%-------------------------------------------------------------------------
% The abstract is the first thing we want to see. No acknowledgements or 
% dedications here. Fetch the abstract from a separate file.
% Please write it in English and in your mother tongue.
% An abstract should be less than half a page, so that both abstracts 
% (that is both languages) fit onto one page.
% We number roman numerals until the main body
\pagenumbering{roman}
% Abstracts are usually written in English, with a version in your
% mother tongue underneath
\chapter*{Abstract} 
\addcontentsline{toc}{chapter}{Abstract}
% Don't change anything above this.
This work is concerned with the calculation of the fundamental group of torus knots. Torus knots are special types of knots which wind around a torus a number of times in the longitudinal and meridional directions. We compute and describe the fundamental group of torus knots by using some concepts in algebraic topology and group theory. We also calculate the fundamental group of an arbitrary knot by using an algorithm called the Wirtinger presentation.

\textbf{Keywords.} Fundamental group, torus knots, arbitrary knots.

\section*{\`Is\d on\'i\d s\'ok\'i (Yor\`ub\'a language)}
I\d s\d{\'e} y\`i\'i j\d {\`e} \`if\d ok\`{an}s\'i p\d {\`e}lu \`i\d sir\`o ti \`aw\d on \d egb\d {\'e} p\`at\`ak\`i \`it\`ak\`un t\'or\d {\'o}\d{\`o}s\`i. \`It\`ak\`un t\'or\d {\'o}\d{\`o}s\`i j\d {\'e} or\'is\`i\'i kan p\`at\`ak\`i n\'in\'u \`aw\d {o}n  \`it\`ak\`un, \'o m\'aa \'n w\'e m\d{\'o} t\'or\d {\'o}\d{\`o}s\`i n\'igb\`a t\'o n\'iye l\d{\'o}n\`aa l\d ongit\'ud\'in\`a \`ati l\d{\'o}n\`aa m\d erid\'i\'on\`a. A se \`ak\'op\d {\`o} \`ati \`ap\'ej\'uwe \`aw\d on \d egb\d {\'e} p\`at\`ak\`i \`it\`ak\`un t\'or\d {\'o}\d {\`o}s\`i n\'ipas\d {\`e} l\'ilo d\'i\d {\`e} n\'in\'u \`aw\d on \`er\`o in\'u t\d op\d {\'o}l\d {\'o}g\`i ti \d {\`o}g\'ib\'ir\`a \`ati d\'i\d {\`e} n\'in\'u \`aw\d on \`er\`o in\'u \`igb\`im\d {\`o} \d egb\d {\'e}. A t\'un \d se \`i\d sir\`o \`aw\d on egb\'e p\`at\`ak\`i ti \`it\`ak\`un al\'a\`il\d {\'e}gb\d {\'e} n\'ipa l\'ilo ag\d {\'o}r\'id\'i\`im\`u kan t\'i a p\`e n\'i \`igb\'ej\'ade Wirtinger.
% At AIMS you are encouraged to repeat the abstract in your mother tongue
% French, Yoruba, Igbo, Malagasy, Arabic, Amharicetc. Iust write it using LaTeX's special
% characters.
% Arabic students use the arabtex package.
% Amharic use openoffice and export from there and import a figure here.
% Where the words do not exist put the English work in italics, or use mathematical symbols.

% Do not change anything below this except for adding your
% signature and name. And take the message to heart.
\vfill
\section*{Declaration}
I, the undersigned, hereby declare that the work contained in this essay is my original work, and that any work done by others or by myself previously has been acknowledged and referenced accordingly.

% Scan your signature into a small picture called 'signature.png' and insert it
% above your name and the date:
\includegraphics[height=2cm]{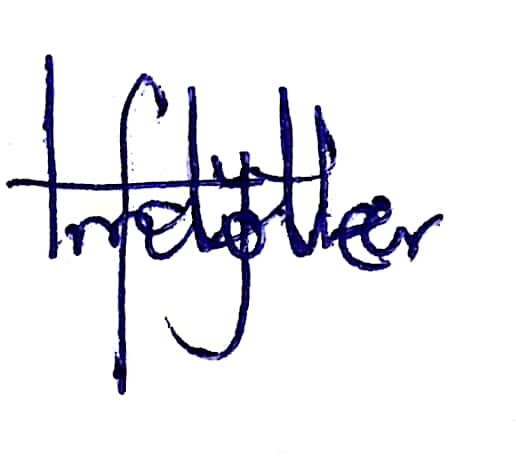} \hrule
% Your name must be in English Capitalisation with no comma, and the Family name comes last. 
% Do note the date below. It is called the "deadline".
Ilyas Aderogba Mustapha, 18 May 2018.

% Don't go typing out the contents.
\tableofcontents
% We strongly *discourage* lists of figures, glossaries, and indices 
% under the AIMS essay style. Get approval before uncommenting this. There
% has to be a good reason for your essay to be navigable in this form, e.g.
% an abnormally large number of figures, or a graphically oriented topic
%\listoffigures
%\addcontentsline{toc}{chapter}{List of Figures}
%-----------------------------------------------------------------------------
\newpage
% We switch to arabic numerals here where your page count starts
% Essays must be 25-35 pages long *starting here* and up to and including
% the conclusion. It does not include the acknowledgements or references.
% 
% Figures may differ between topics, but they are not there
% to fill the pages -- they must add meaning.
% In general most figures should be 0.8 times the width of the page
% (perhaps wider when two or three columns of figures)
% See the example in chapter three for defining that. Be *consistent*
% in your presentation of information.
\pagenumbering{arabic}
\pagestyle{myheadings}
%-----------------------------------------------------------------------------
% Each chapter goes in a separate file
% Chapter titles are just examples
% Always have a question
% Note the Case Pattern used at AIMS
\chapter{Introduction}
\label{chap1}
A knot is an embedding of the unit circle $S^1$ inside the 3-dimensional Euclidean space, $\mathbb{R}^3$. A torus knot is a special type of knot which lies on the surface of a torus. Specifically, consider the torus $S^1\times S^1$ embedded in $\mathbb{R}^3$ in the standard way. Let the map $f:S^1\to S^1\times S^1\subset \mathbb{R}^3$ be defined as $f(z)=(z^m,z^n)$, where $m$ and $n$ are relatively prime positive integers. One can show that the map $f$ is an embedding (see Proposition \ref{c3p302}). The torus knot $K=K_{m,n}$ is defined to be the image of $S^1$ under the map $f$. The knot $K$ wraps the torus a total of $m$ times in the longitudinal direction and $n$ times in the meridional direction. Figures \ref{fig1.1} and \ref{fig1.2} show the cases when $(m,n)=(2,3)$ and $(m,n)=(3,4)$ respectively.
\begin{figure}[htbp!]
	\begin{center}
		\begin{minipage}[b]{0.5\linewidth}
			\centering
			\includegraphics[width=0.6\linewidth]{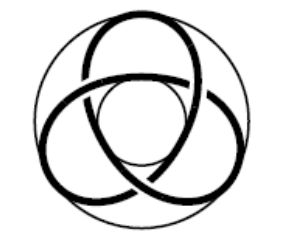} 
			\caption{$(m,n)=(2,3)$ \cite{allen}.}\label{fig1.1} 
		\end{minipage}\hspace{-40pt}%%
		\begin{minipage}[b]{0.5\linewidth}
			\centering
			\includegraphics[width=0.55\linewidth]{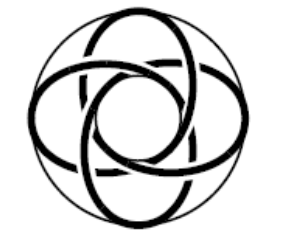} 
			\caption{$(m,n)=(3,4)$ \cite{allen}.}\label{fig1.2}
		\end{minipage}
	\end{center}   
\end{figure}
\vspace{-2pt}
When $m=2$ and $n=3$ as in Figure \ref{fig1.1}, the knot $K=K_{2,3}$ is called a trefoil knot. We can also allow negative values for $m$ or $n$, but it will change the knot $K$ to a mirror-image knot.

This work is concerned with the calculation of the fundamental group of $K$ (that is, the fundamental group of its complement, $\mathbb{R}^3\setminus K$, in $\mathbb{R}^3$), and the description of the structure of that group. Also, we calculate the fundamental group of an arbitrary knot $L$, by using Wirtinger presentation.
\section{Objectives of the Study}
\label{c1s1}
The objectives of this work are the following:
\begin{enumerate}
	\item Compute the fundamental group of $\mathbb{R}^3\setminus K$, where $K$ is a torus knot.
	\item Describe the structure of the fundamental group of $\mathbb{R}^3\setminus K.$
	\item Compute the fundamental group of $\mathbb{R}^3\setminus L$, where $L$ is an arbitrary knot.
\end{enumerate}
\vspace{-9pt}
\section{Structure of the Essay}
\vspace{-2pt}
This work is divided into four chapters. In the first chapter, we introduce the topic, state the objectives, and briefly review the literature. In Chapter \ref{chap2}, we study some concepts in group theory and fundamental group. These concepts will be of great help in the understanding of this work. 
In Chapter \ref{chap3}, we compute and describe the fundamental group of $\mathbb{R}^3\setminus K$. Also, we make use of the Wirtinger presentation to compute the fundamental group of $\mathbb{R}^3\setminus L$, for an arbitrary knot $L$. Finally, in Chapter \ref{chap4}, we give the conclusion of the work.
\newpage
\section{Brief Review of Literature}
Many authors have discussed extensively on the theory of knot. In the work of Martin Sondergaard \cite{chris}, he introduced basic knot theory, along with concepts from topology, algebra, and algebraic topology, as they relate to knot theory. Then, Chiara and Renzo \cite{obert} presented geometric and topological properties of torus knots. They used standard parametrization and symmetric aspects to find new results on local and global properties of the knots.

In the shape of torus knots, elastic filaments have been studied in relation to bending and torsional energy \cite{coleman}. Examples of minimum energy configurations were also presented in \cite{coleman} for both torus knots and chain formed by linking two unknots. Also, in terms of Mobius energy, Kim and Kusner \cite{kim} used the principle of symmetric criticality to construct torus knots and links that extremize the Mobius invariant energy which was introduced by O'Hara and Freedman. It was described further in \cite{kim}, the experiments with a discretized version of the mobius energy which can be applied to the study of arbitrary knots and links. 

Torus knots have been found by Ricca \cite{ric} in the theory of integrable systems as solutions to differential equations. He went further to use Hasimoto map to interpret soliton conserved quantities in terms of global geometric quantities and he showed how to express these quantities as polynomial invariants for torus knots. Also, in Maxwell's theory of electromagnetism, Kedia et al \cite{ked} constructed a family of null solutions to Maxwell's equation in free space whose field lines encode all torus knots and links. They went further to illustrate the geometry and evolution of the solutions and manifested the structure of nested knotted tori filled by the field lines.

As phase singularities in optics, Irvine and Bouwmeester \cite{irv} showed how a new class of knotted beams of light can be derived, and showed that approximate knots of light may be generated using tightly focused circularly polarized laser beams. In fluid mechanics, Ricca et al \cite{ricca} addressed the time evolution of vortex knots in the context of Euler equations. They went further and found that thin vortex knots which are unstable under the localized induction approximation have a greatly extended lifetime whenever Biot-Savart induction law is used. Torus knots may also arise as colloids as stated in \cite{tka}, and in many other physical, biological and chemical context.

 % Introduction is usually a chapter itself.
\chapter{Preliminaries}
\label{chap2}
In this chapter, we discuss some important concepts which will allow us to understand this work better. 
In the first section, we recall some concepts in group theory. Then we introduce the notion of free groups and presentations in Section \ref{c2s2}. Section \ref{c2s3} deals with the free product of groups. Furthermore, in Section \ref{c2s4}, we give an introduction to the fundamental group of topological spaces, and we see what it means for a topological space to deformation retract onto one of its subspaces in Section \ref{c2s5}. Finally, in Section \ref{c2s6}, we state van Kampen's theorem.  This theorem is useful in the computation of the fundamental group of a topological space which can be decomposed into simpler spaces with known fundamental groups.
\vspace{-7pt}
\section{Introduction to Group Theory}\label{c2s1}
The goal of this section is to recall some concepts in group theory, such as normal subgroups, homomorphisms, kernels and quotients, abelianization of a group and many more.
\begin{defn}
	A nonempty set $G$ together with a binary operation
	$$(x,y)\mapsto x\ast y : G\times G\rightarrow G,$$
	is called a group if the operation is associative, there exists an identity element, and every element of $G$ is invertible.\label{s2d1}
\end{defn}	
	Sometimes, we write $G$ instead of $(G,\ast)$ and $xy$ instead of $x\ast y$. Examples are
		$(\mathbb{Z},+), (\mathbb{R},+),$ and $ (\mathbb{R}\setminus \{0\},\cdot)$. However, $(\mathbb{R},\cdot)$ is not a group since the inverse of 0 does not exist.
	\begin{defn}
		The order of a group $G$, denoted by $|G|$ or $\mathrm{ord}(G)$, is the cardinality of $G$.
	\end{defn}
	\begin{defn}
		Let $G$ be a group and let $x\in G$. Then the order of $x,$ denoted by $\mathrm{ord}(x)$ or $|x|$, is the smallest integer $m>0$ such that $x^m=e$. 		
		If such $m$ does not exist, then $x$ is said to have infinite order.
	\end{defn}
	\begin{defn}[Direct product of groups]
		Let $(G_1,\ast)$ and $(G_2,\diamond)$ be groups, then $(G_1\times G_2,\cdot)$,
		 called the direct product of $G_1$ and $G_2$, is a group with the binary operation $(\cdot)$ defined component-wise.
		 That is,
		$$(x_1,y_1)\cdot(x_2,y_2)=(x_1\ast x_2, y_1\diamond y_2),~\forall x_1, x_2\in G_1\text{ and }\forall y_1,y_2\in G_2,$$
		where the Cartesian product, $G_1\times G_2$ is the underlying set.	
	\end{defn}
	\begin{defn}
		A group $G$ is said to be abelian or commutative if for all $ x,y\in G$, $xy=yx.$
	\end{defn}
		If a group is abelian, then direct product can be regarded as direct sum, and it is denoted by $G_1\oplus G_2$.
	\begin{defn}
		Let $G$ be a group, then the \textbf{torsion} of $G$, denoted by $\mathrm{Tor}(G)$, is defined as $$\mathrm{Tor}(G)=\{x\in G| ~x^n=e, n\in \mathbb{N}\}.$$
	\end{defn}
	If all elements of a group $G$ has infinite order except the identity element, then the group is said to be \textbf{torsion-free}.
	\begin{defn}
		The center of a group $G$, denoted by $Z(G)$, is defined as $$Z(G)=\{x\in G| xy=yx,~ \forall y\in G\}.$$
	\end{defn}
	\subsection{Subgroups}
	Suppose $(G,\ast)$ is a group. All the subsets of $G$ that form a group under $\ast$ are called 
	\newpage
	subgroups of $G$. Subgroup is usually denoted by $"\le".$ If $G_1$ is a subgroup of $G_2$, we write $G_1\le G_2.$
	\begin{defn}\cite{john}:
		Let $G$ be a group, and let $\emptyset\ne S\subseteq G$. Then $S$ is a subgroup of $G$ if
		\begin{itemize}
			\item  $e\in S$;
			\item $\forall x, y\in S, ~xy\in S$;
			\item for each $x\in S, ~x^{-1}\in S,$ where $x^{-1}$ is the inverse of $x$.
		\end{itemize}\label{sub}
	\end{defn}
	\begin{pro}\cite{john}:
		Let $G$ be a group, and let $H\subseteq G$. Then $H$ is a subgroup of $G$ if and only if $H\ne \emptyset$, and for any $x, y\in H,~xy^{-1}\in H.$ \label{s2p1}
	\end{pro}
	The center of a group $G$, $Z(G)$, is a subgroup of $G$. In fact, it is a normal subgroup of $G$ as shown in Example \ref{s2e1}.
	\begin{defn}\cite{milneGT}:
		Let $X$ be a subset of a group $G$. Then $\langle X\rangle$ is called the subgroup generated by $X$, where $\langle X\rangle$ is the set of all finite products of elements of $X$. That is $\langle X\rangle=\left\{\displaystyle\prod_{i=1}^{n}a_i|~a_i\in X\right\}.$
	\end{defn}
	The set $X$ is a basis of group $G$, and hence the rank of $G$ is the cardinality of $X$.
	\begin{defn}\cite{milneGT}
		A group $G$ is said to be cyclic if it is generated by a single element. That is, $$G=\langle x\rangle =\{x^n|n\in\mathbb{Z}\}, \text{ for some } x\in G.$$
	\end{defn}
	If the operation defined on $G$ is addition (+), then $G=\langle x\rangle=\{nx|n\in\mathbb{Z}\}$ as in the example below.
		\begin{exa}
			$(\mathbb{Z}, +)$ is a cyclic group since it is generated by 1. That is $\mathbb{Z}=\langle 1\rangle=\{n.1|n\in\mathbb{Z}\},$ where $n\cdot 1=1+1+\dots+1$ ($n$-times).
		\end{exa}
	Now, suppose we want to find the order of an element of a finite group, the following proposition says that it is enough to find the cardinality of the cyclic group generated by that element.
	\begin{pro}\cite{john}:
		Let $G$ be a finite group, and let $x\in G.$ Then $\mathrm{ord}(x)=|\langle x\rangle|$.\label{s2p2}
	\end{pro}
	\begin{defn}[Group homomorphism]\cite{john}:
		Let $(G,\ast)$ and $(H,\diamond)$ be two groups. A {group homomorphism} from $G$ to $H$ is a map $f: G\rightarrow H$ such that $$f(x\ast y)=f(x)\diamond f(y), ~\forall x,y\in G.$$
	\end{defn}
		If further, $f$ is bijective, then we say that $f$ is an \textbf{isomorphism}. In this case, we say that $G$ is isomorphic to $H$, and this is written as $G\cong H.$ 
	\begin{defn}[Cosets]\cite{john}:
		Let $H$ be a subgroup of a group $G$, and let $x\in G$. Then the \textbf{coset} $xH$ is a subset of $G$ defined as $$xH=\{xh| h\in H\}.$$
	\end{defn}
		The set $xH$ as defined above is called left coset of $H$ in $G$. We also have right coset of $H$, which are subsets of the form $Hx=\{hx|h\in H\}.$ So a coset can either be left coset or right coset.
	\begin{pro}\cite{conrad}:
		Let $H$ be a subgroup of a group $G$. Then $xH=H$ if and only if $x\in H.$\label{s2p17}
	\end{pro}
	\newpage
	\begin{defn}\cite{milneGT}:
		Let $H$ be a subgroup of a group $G$. Then the index $(G:H)$ of $H$ in $G$, is the cardinality of the set $\{aH| a\in G \}.$
	\end{defn}
	\begin{exa}
		If $G$ is a group, then $(G:1)$ is the order of $G$. In this case, $H=1$, and hence we have that $(G:1)$ is the cardinality of the set $\{a|a\in G\}$, which is the order of $G$.
	\end{exa}
	\begin{thm}[Lagrange's Theorem \cite{milneGT}]
		Let $G$ be a finite group, and let $H\le G$. Then $$(G:1)=(G:H)(H:1).$$ 
	\end{thm}
	\begin{cor}\cite{john}:
		Let $G$ be a finite group, and let $x\in G$. Then $\mathrm{ord}(x)$ divides $\mathrm{ord}(G).$\label{s2c1}
	\end{cor}
	\begin{proof}
			By definition of $\langle x\rangle$, we have that $\langle x\rangle$ is a subgroup of $G$.
			Also, by Lagrange's theorem, $(G:1)=(G:H)(H:1)$. That is $$\mathrm{ord}(G)=t\cdot\mathrm{ord}(\langle x\rangle)=t \cdot\mathrm{ord}(x),$$ where $t=(G:H).$ Hence $\mathrm{ord}(x)$ divides $\mathrm{ord}(G).$
	\end{proof}
	\begin{defn}[Normal Subgroup \cite{milneGT}]
		Let $G$ be a group, and let $N$ be a subgroup of $G$. Then $N$ is a \textbf{normal subgroup} of $G$, denoted by $N\triangleleft G$, if and only if $gNg^{-1}=N$, for all $g\in G.$
	\end{defn}
	    This means that a normal subgroup of a group $G$ is a subgroup which does not change under conjugation by elements of $G.$
		To show that $N$ is a normal subgroup, it is enough to show that $gNg^{-1}\subseteq N$, for all $g\in G$. That is, $gng^{-1}\in N$ for all $g\in G$, and for all $n\in N.$ This is because $N\subseteq gNg^{-1}$ always, this can be seen by taking $g$ to be the identity element of $G$.
	\begin{exa}
		The center of a group $G$, $Z(G)$, is a normal subgroup of $G$. This is because if $g$ is an arbitrary element in $G$, and $z\in Z(G)$, then $$gzg^{-1}=zgg^{-1}\text{ (since $z\in Z(G)$)}=z\in Z(G).\label{s2e1}$$
	\end{exa}
	\subsection{Kernels and Quotients}
	Suppose $G$ and $H$ are two groups, and $f$ is a group homomorphism from $G$ to $H$. Then the kernel of $f$, is the set of all elements in $G$ that are mapped to the identity element in $H$.
	\begin{defn}\cite{milneGT, john}:
		Let $f:G\to H$ be a group homomorphism. Then the kernel of $f$, denoted by $\ker(f)$ is
		$$\ker(f)=\{x\in G| f(x)=e\},$$
		and the image of $f$, denoted by Im$(f)$ is $$\text{Im}(f)=\{f(x)\in H| x\in G\}.$$
	\end{defn}
	\begin{defn}
		Let $G$ be a group, and let $N\triangleleft G$. Then the quotient group (also known as factor group) of $N$ in $G$, written as $G/N$ or $\dfrac{G}{N}$ is the set of all cosets of $N$ in $G$. That is  $$G/N=\{gN | ~g\in G\}.$$
	\end{defn}
	All elements of $G/H$ are of the form $gH$ for some $g\in G$, and the identity element of $G/N$ is $N$.
	\begin{defn}\cite{john}
		Let $G$ be a group and let $x\in G$. Then $y\in G$ is said to be a conjugate of $x$ if $$y=gxg^{-1}, \text{ for some $g\in G.$}$$
	\end{defn}
	\begin{thm}[First Isomorphism Theorem \cite{john}]
		If $\phi : G\rightarrow H$ is a homomorphism. Then
			 $$\ker \phi \triangleleft G \text{ and }\frac{G}{\ker \phi} \cong  \text{Im}(\phi).$$
	\end{thm}
	\subsection{Commutator Subgroup}
	The commutator subgroup of a group $G$, also known as derived
	subgroup of $G$, is important because it is the smallest normal subgroup such that the factor group of $G$ by this subgroup is abelian.
	\begin{defn}
		Let $G$ be a group, and let $x,y\in G.$ Then the \textbf{commutator} of $x$ and $y$ is $xyx^{-1}y^{-1}.$
	\end{defn}
	\begin{defn}\cite{wolf2}
		Let $G$ be a group. The commutator subgroup of $G$ is the subgroup generated by the commutators of its elements. It is denoted by $[G,G].$
	\end{defn}
	\subsection{Abelianization}
	Generally, groups are not abelian. But there is always a group homomorphism $\phi: H\rightarrow H'$ to an Abelian group. This homomorphism is called abelianization. The homomorphism is described by its kernel, the commutator subgroup $[H,H]$, which is the unique smallest normal subgroup of $H$ such that $H'=H/[H,H]$ is Abelian \cite{wolf3}.
	\begin{exa}
		Let $G$ be a group generated by $x$ and $y$. Then the Abelianization of $G$ is given as $$G'=\frac{G}{[G,G]}=\frac{\langle x,y\rangle}{\langle xyx^{-1}y^{-1}\rangle}.$$
	\end{exa}
	In short, the abelianization of a group $G$, is the quotient group $\dfrac{G}{[G,G]}$, where $[G,G]$ is the commutator subgroup of $G$.
	\section{Free Groups and Presentations}\label{c2s2}
	This section is about group presentation, $\langle X|R\rangle$, where $X$ is the set of all generators of the group, and 
	$R$ is the set of the relations. Naturally, many groups arise in this form in algebraic topology, so we need to recall them. A group of the form $\langle X|\rangle$, that is, no relation, is called a free group.	
	\begin{defn}
		Let $X$ be a set. A \textbf{word} over $X$ is a sequence $w=w_1w_2\dots w_n, ~n\ge 0,~w_i\in X.$ In this case, the length of the word is $n$. When $n=0,$ we have an empty word, denoted by $e.$
	\end{defn}
	\begin{defn}
		Let $A$ and $B$ be two sets. The disjoint union of $A$ and $B$, denoted by, $A\amalg B$ is the set $$(A\times \{0\})\cup (B\times \{1\}).$$
	\end{defn}
	If $A$ and $B$ are groups, we define $A\star B$ to be the set of all words $w_1\dots w_n,~n\ge 0$, on the alphabet $A\amalg B$ such that adjacent letters $w_i$ and $w_{i+1}$ are from different groups. Then words of this form are called \textbf{reduced words}.
	\begin{exa}
		Let $\mathbb{Z}_x=\langle x\rangle$, and $\mathbb{Z}_y=\langle y\rangle$. The word $x^2yx^{-1}x^6$ is not reduced, and hence the reduction of this word will be $x^2yx^5\in\mathbb{Z}_x\star\mathbb{Z}_y$.
	\end{exa}
	\begin{defn}\cite{wolf4}:
		A group $F$ is said to be \textbf{free} if there is no relation between its generators except the relationship between an element and its inverse.
	\end{defn}
		This means that, in free groups, relators are of the form $abb^{-1}a^{-1}$.
	\begin{defn}\cite{holt2013presentation}:
		Let $X$ be a subset of a group $F$. Then $F$ is free on $X$ if for any group $G$ and any map $f:X\rightarrow G,$ there exists uniquely a homomorphism $f':F\rightarrow G$ with $f'(x)=f(x),~\forall x\in X.$
	\end{defn}
	\begin{pro}\cite{holt2013presentation}:
		Let $G$ be any group. Then $G$ is isomorphic to a quotient group of a free group.\label{s2p3}
	\end{pro}
	\begin{proof}
			 Let $A\subseteq G$ such that $G=\langle A\rangle.$ Then the map $f:A\rightarrow G$ extends to $f': F_A\rightarrow G$, where $F_A$ is a free group on $A$.
			 The group $G=\langle A\rangle$ implies that $f'$ is surjective.
			Hence, by first isomorphism theorem, $F_A/\ker f'\cong G.$ 
	\end{proof}
	\begin{defn}\cite{johnson1997presentations}
		Let $X$ be a set, $F$ the free group on $X$, $R\subseteq X$, $N$ a normal closure of $R$ in $F$, and $G$ the quotient group $F/N$. Then $G=\langle X|R\rangle$ is a \textbf{presentation} of $G$. Elements of $X$ are the generators and elements of $R$ are called the relators.
	\end{defn}
	\begin{defn}\cite{johnson1997presentations}
		Let $G$ be a group. Then $G$ is finitely presented if it has a presentation with both $X$ and $R$ finite.
	\end{defn}
	\begin{rem}\cite{johnson1997presentations}:
		Sometimes, we replace $R$ in $\langle X|R\rangle$ by $R=1$, i.e $\{r=1|r\in R\}$, called defining relations for $G$.
	\end{rem}
	\begin{exa}
		\begin{enumerate}
			\item Let $A$ be a set. Then $\langle A|~\rangle$ is a presentation of the free group of rank $|A|.$
			\item $\langle x,y|x^3,y^4\rangle$ can also be written as $\langle x,y|x^3=1, y^4=1\rangle$.
			\item $\dfrac{\langle a,b| a^4, b^3\rangle}{\langle (ab)^2\rangle}=\langle a,b|a^4, b^3, (ab)^2\rangle = \langle a,b|a^4=b^3=1, ab=b^{-1}a^{-1}\rangle$.
			\item $\mathbb{Z}=\langle x|~\rangle,$ and $\mathbb{Z}_n=\langle x|x^n=1\rangle$. The reason why the group $\langle x|x^n=1\rangle=\mathbb{Z}_n$, is because $\langle x|x^n=1\rangle=\{1,x,x^2,x^3,\dots,x^{n-1}\}=\mathbb{Z}_n.$
		\end{enumerate}
	\end{exa}
	\begin{exa}\cite{holt2013presentation}:
		$\langle x, y |[x,y]\rangle = \langle x, y|xy=yx\rangle\cong\mathbb{Z}\times\mathbb{Z}$.
		In general, if there are $n$ generators, then the group $$\langle x_1,\dots,x_n |[x_i,x_j], 1\le i<j\le n\rangle \cong \mathbb{Z}^n.$$
	\end{exa}
	\begin{pro}\cite{holt2013presentation}
		All groups have presentations.
	\end{pro}
	\begin{proof}
			Let $G$ be a group, $A\subseteq G$ be a set of generators of $G$, and $f':F_A\rightarrow G$. 
			Then by Proposition \ref{s2p3}, we have that $$G\cong F_A/\ker f'.$$
			Let $R\subset \ker f'$ such that $\ker f'=\langle R\rangle$. Then, we have that  $G\cong \dfrac{F_A}{\ker f'}=\dfrac{\langle A\rangle}{\langle R\rangle}=\langle A|R\rangle.$ 
	\end{proof}
	\section{Free Product of Groups}\label{c2s3}
	The operation which takes two groups $H_1$ and $H_2$ to form a new group $H_1\star H_2$ is called free product. It is important in algebraic topology because of Van kampen's theorem, which expresses the fundamental group of the union of two open and path-connected spaces whose intersection is also open and path-connected as a combined free product of the fundamental groups of the spaces.
	\begin{defn}
		Let $G_1$ and $G_2$ be groups. Then the \textbf{free product} of $G_1$ and $G_2$, denoted by $G_1\star G_2$ is defined as 
		$$G_1\star G_2=\{x|x \text{ is a reduced word in the alphabet }G_1\amalg G_2\}.$$
	\end{defn}
	We define the binary operation $(\cdot)$, also know as concatenation, on $G_1\star G_2$ as follows 
	\begin{eqnarray*}
		\cdot : G_1\star G_2\times G_1\star G_2&\rightarrow& G_1\star G_2\\
		(w,z)&\mapsto&w\cdot z=wz.
	\end{eqnarray*}
	\begin{pro}
		$(G_1\star G_2,\cdot)$ is a group.
	\end{pro}
	\begin{proof}
		First of all, we have to show that $G_1\star G_2\ne\emptyset$, then we will show that $(G_1\star G_2,\cdot)$ satisfies all the conditions stated in Definition \ref{s2d1}.
			The set $G_1\star G_2\ne \emptyset$ since $G_1$ and $G_2$ are not empty because they are groups.
			Now let $x=x_1\dots x_r,~ y = y_1\dots y_n,$ and $z=z_1\dots z_s\in G_1\star G_2.$ We want to show that $(x\cdot y)\cdot z=x\cdot(y\cdot z)$. This is a bit tricky because we do not know whether $x_r y_1$ and $y_nz_1$ are reduced or not. Without loss of generality, we suppose that $x_ry_1$ is reduced and $y_nz_1$ is not reduced. Let the reduction of $x_r y_1$ be $w$. Then 
			\begin{eqnarray*}
				(x\cdot y)\cdot z=(x_1\dots x_ry_1\dots y_n)\cdot z_1\dots z_s=x_1\dots x_{r-1}wy_2\dots y_nz_1\dots z_s.
			\end{eqnarray*}
		So, we have that
		\begin{eqnarray*}
			x\cdot(y\cdot z)&=&x_1\dots x_r\cdot(y_1\dots y_nz_1\dots z_s)\\
			&=&x_1\dots x_ry_1\dots y_n z_1\dots z_s\\
			&=&x_1\dots x_{r-1}wy_2\dots y_nz_1\dots z_s\\
			&=&(x\cdot y)\cdot z.
		\end{eqnarray*}
		The identity element here is the empty word, $e$, since for any $x\in G_1\star G_2$, $x\cdot e = e\cdot x=x.$
		
		Let $w=w_1\dots w_r\in G_1\star G_2$ be arbitrary, and let $w^{-1}$ be the inverse of $w$. Then by definition,
		$$w\cdot w^{-1}=w_1\dots w_r\cdot w^{-1}=e\implies w^{-1}=w_r^{-1}\dots w_1^{-1}.$$
		
		Thus for any $w=w_1\dots w_r\in G_1\star G_2,$ the inverse of $w$ is $w_r^{-1}\dots w_1^{-1}.$ Hence $G_1\star G_2$ is a group.
	\end{proof}	
	\begin{pro}\cite{allen}
		Let $A$ and $B$ be groups, and let $x\in$ $\mathrm{Tor}(A\star B)$, then $$x=yzy^{-1},$$ where $z\in$ $\mathrm{Tor}(A)$ or $\mathrm{Tor}(B)$, and $y\in A\star B.$\label{s23p3}
	\end{pro}
	This means that all torsion elements of the group $A\star B$ are conjugates of torsion elements of $A$ or $B$.
	\begin{proof}
		We will prove this by induction on the length of  the word. 
			It is true for empty word( that is, a word of length 0), since $$e_T=e_{G}ee_G^{-1},$$
			where $e_T$ is the empty word in $\mathrm{Tor}(A\star B)$, $e_G$ is the empty word in $A\star B$, and $e$ is the empty word in $\mathrm{Tor}(A)$ or $\mathrm{Tor}(B)$.
			
			It is also true for a word of length one since, for example if $w_1\in$ $\mathrm{Tor}(A\star B)$ is a word of length one, then $w_1\in A$ or $B$, and $$w_1^p=e_G,\text{ for some $p\in \mathbb{N}$.}$$
			
			Also $w_1$ can be written as $e_Gw_1e_G^{-1}.$ Since $w_1^p=e_G$, then we have that $w_1\in$ $\mathrm{Tor}(A)$ or $\mathrm{Tor}(B).$
			\newpage			
			Now, we assume that $w=w_1\dots w_n\in$ $\mathrm{Tor}(A\star B)$ is of length $n>1$ and that the statement holds for all words of length $k<n.$ Since $w\in$ $\mathrm{Tor}(A\star B)$, then $$w^q=e_G,\text{ for  some $q\in\mathbb{N}.$}$$ This implies that $w_nw_1=e\implies w_n=w_1^{-1}$.
			Hence $w$ must be of the form $w_1w_2\dots w_{n-1}w_1^{-1}.$
			But then $w_2\dots w_{n-1}$ is a word of length $n-2<n$. Hence by induction hypothesis,
			$$w_2\dots w_{n-1}=vgv^{-1},\text{ for some $g\in$ $\mathrm{Tor}(A)$ or $\mathrm{Tor}(B)$, and $v\in A\star B.$}$$
			\item Thus $w=w_1vgv^{-1}w_1^{-1}=w_1vg(w_1v)^{-1}.$ Hence the statement holds.
	\end{proof}
	Having studied these concepts in group theory, we now discuss some important concepts in algebraic topology which would be needed to compute the fundamental group of torus knots which we will see in the next chapter. 
	\section{Introduction to Fundamental Group}\label{c2s4}
	This section gives an introduction to the fundamental groups of a topological spaces. We begin by introducing the terminologies adopted and then a few general result is stated.
	\begin{defn}
		Let $X$ and $Y$ be topological spaces, and let $f:X\rightarrow Y$ be a bijection. If both $f$ and the inverse of $f$, $f^{-1}$, are continuous, then $X$ is said to be homeomorphic to $Y$, written as $X\cong Y$.
	\end{defn}
	The condition that $f$ is continuous means that for all open set $U$ of $Y$, $f^{-1}(U)$ is open in $X$. 
	\begin{pro}\cite{chris}:
		Let $p\in S^n$ be $(0,0,\dots,1)\in\mathbb{R}^{n+1}$, where $S^n=\{(x_1,\dots,x_{n+1})\in\mathbb{R}^{n+1}|x_1^2+\dots+x_{n+1}^2=1\}$. Then the stereographic projection $\phi:S^n\setminus p \rightarrow\mathbb{R}^n$ given by $$\phi(x_1,\dots,x_n,x_{n+1})=\frac{1}{1-x_{n+1}}(x_1,\dots,x_n),$$ is a homeomorphism.\label{s2p5}
	\end{pro}
	\subsection{Homotopy Type}
	In general topology, $X=Y$ if $X\cong Y$, where $X$ and $Y$ are topological spaces  but in algebraic topology, we have a weaker notion of equality called homotopy equivalence, denoted by $``\simeq"$. One reads $X\simeq Y$ as $X$ is homotopy equivalent to $Y$.
	\begin{defn}\cite{weng}
		Let $X$ and $Y$ be topological spaces and let $f,g:X\rightarrow Y$ be continuous maps. We say that $f$ is \textbf{homotopic} to $g$ and we write $f\sim g$, if there exists a continuous map $F:X\times [0,1]\rightarrow Y$ such that $$F(x,0)=f(x)\text{ and }F(x,1)=g(x), \text{ for all } x\in X.$$
	\end{defn}
	The continuous map $F$ is called a homotopy from $f$ to $g$. We can consider a homotopy as a continuous deformation of the map $f$ to the map $g$, as $t$ denotes the time from 0 to 1.
	\begin{defn}
		Let  $X$ and $Y$ be topological spaces, and let $f:X\rightarrow Y.$ We say that $f$ is  \textbf{homotopy equivalence} if there exists another map $g:Y\rightarrow X$ such that $gf\sim \mathrm{Id}_X$ and $fg\sim \mathrm{Id}_Y.$
	\end{defn}
	The spaces $X$ and $Y$ are homotopy equivalence (or have the same homotopy) if there is a homotopy equivalent between them, and we write $X\simeq Y.$
	\begin{exa}
		The unit disk $D^n\simeq \{(0,0,\dots,0)\},~n\ge 0$.		
		 To see this, we define the map $f:D^n\rightarrow \{(0,0,\dots,0)\}$ by $f(x,y)=(0,0,\dots,0)$, and $g:\{(0,0,\dots,0)\}\rightarrow D^2$ by $g(0,0,\dots,0)=(0,0,\dots,0)$. 
		 Of course $fg=\mathrm{Id}_{\{(0,0,\dots,0)\}}$ because $fg(0,0,\dots,0)=f((0,0,\dots,0))=(0,0,\dots,0).$ 
		  We define the homotopy from $gf$ to $\mathrm{Id}_{D^n}$ by   $F(x,t)=(1-t)gf(x)+tx,$ for all $x\in D^n$. 
		  
		 Since $F(x,0)=gf(x)$, and $F(x,1)=x=\mathrm{Id}_{D^n}(x)$, then we have that $gf\simeq \mathrm{Id}_{D^n}$. Therefore $D^n\simeq \{(0,0,\dots,0)\}.$
		 
		 Similarly, for all $n\ge 0$, $\mathbb{R}^n\simeq\{(0,0,\dots,0)\}.$ \label{c2e1}
	\end{exa} 
	\begin{defn}
		A topological space $Y$ is \textbf{contractible} if it is homotopy equivalent to a point, $pt.$
	\end{defn}
	As we have seen in Example \ref{c2e1} above, $D^n\simeq pt=\{(0,0\dots,0)\}$. Hence the space $D^n$ is contractible. Similarly, $\mathbb{R}^n$ is contractible for all $n\ge 0.$
	\begin{defn}
		Let $Y$ be a topological space. A \textbf{path} in $Y$ is a continuous map $\alpha :[0,1]\rightarrow Y$. The starting point is $\alpha(0)$ while the ending point is $\alpha(1)$. If $\alpha(0)=\alpha(1)$, the path becomes a \textbf{loop}.
	\end{defn}
	\begin{defn}
		Let $Y$ be a topological space and let $\alpha,\beta:[0,1]\rightarrow Y$ be paths in $Y$ such that $\alpha(1)=\beta(0)$. Then the path $\alpha\cdot\beta$ defined by 
		$$
		(\alpha\cdot\beta)(s)=\begin{cases}
			\alpha(2s) \text{ if } s\in[0,\frac{1}{2}]\\
			\beta(2s-1) \text{ if } s\in[\frac{1}{2},1].
		\end{cases}
		$$ is called the \textbf{composition} of $\alpha$ and $\beta.$ Geometrically, $\alpha\cdot\beta$ can be viewed as the concatenation of paths $\alpha$ and $\beta.$
	\end{defn}
	\begin{defn}
		Let $y\in Y$. The path $c_y:[0,1]\rightarrow Y$ defined by $c_y(s)=y,~\forall s\in [0,1]$ is called the {constant} path at $y$.
	\end{defn}
	\begin{lem}\cite{paulst}:
		Let $x,y\in X$, and let $\alpha:[0,1]\to X$ be a path such that $\alpha(0)=x$, and $\alpha(1)=y.$ Then $\alpha\sim c_x\cdot \alpha$ and $\alpha\sim \alpha\cdot c_y$. \label{c2l2411}
	\end{lem}
	\begin{defn}
		Let $\alpha:[0,1]\rightarrow Y$ be a path. The path $\bar{\alpha}:[0,1]\rightarrow Y$ defined as $\bar{\alpha}(s)=\alpha(1-s),$ for all $s\in [0,1]$ is called \textbf{inverse path} of $\alpha.$
	\end{defn}
	The fundamental group is defined in terms of loops and deformation of loops. More generally, it is useful sometimes to consider paths and deformation of paths. The idea of deforming a path continuously in which its endpoint is fixed is given in the following definition.
	\begin{defn}
		Let $\alpha,\beta:[0,1]\rightarrow Y$ be paths such that $\alpha(0)=\beta(0)$ and $\alpha(1)=\beta(1)$. Then $\alpha$ is homotopic to $\beta$ if there exists a continuous map $F:[0,1]\times [0,1]\rightarrow Y$ such that 
		\begin{itemize}
			\item $F(s,0)=\alpha(s)$ and $F(s,1)=\beta(s), ~\forall s.$
			\item $F(0,t)=\alpha(0)$ and $F(1,t)=\alpha(1),~\forall t.$
		\end{itemize}	
	\end{defn}
	\begin{lem}\cite{paulst}:
		Let $\alpha, \beta, \alpha', \beta': [0,1]\to X$ be paths such that $\alpha(0)=\alpha'(0),~\alpha(1)=\alpha'(1)=\beta(0)=\beta'(0),$ and $\beta(1)=\beta'(1)$. If $\alpha\sim \alpha'$ and $\beta\sim\beta'$, then $\alpha\cdot\beta\sim \alpha'\cdot\beta'.$ \label{c2l2413}
	\end{lem}
	\begin{lem}\cite{paulst}:
		Let $y\in Y$ and let $\alpha :[0,1]\rightarrow Y$ be a path such that $\alpha(0)=x$. Then $\alpha\cdot\bar{\alpha}\sim c_x$ and $\bar{\alpha}\cdot \alpha\sim c_x$.\label{c2l2415} 
	\end{lem}
	Let $Y$ be a topological space and let $y\in Y$ be the basepoint of a loop $\alpha$. We define $\pi_1(Y,y)$ to be the set of homotopy classes $[\alpha]$ of loops in $Y$ at $y$.
	\begin{lem}\cite{paulst}:
		Let $X$ be a topological space, and let $\alpha, \beta, \gamma:[0,1]\to X$ be paths in $X$ such that $\alpha(1)=\beta(0),$ and  $\beta(1)=\gamma(0).$ Then $\alpha\cdot(\beta\cdot\gamma)\sim(\alpha\cdot\beta)\cdot\gamma$.\label{c2e2415}
	\end{lem}
	\begin{pro}\cite{allen, paulst}:
		The set $\pi_1(Y,y)$ is a group with respect to the operation $[\alpha]\cdot [\beta]=[\alpha\cdot\beta]$.
	\end{pro}
	\begin{proof}
		First of all, let us show that the product $[\alpha]\cdot[\beta]=[\alpha\cdot\beta]$ is well defined. Suppose that $\alpha, \alpha', \beta, \beta'$ are loops in $Y$ such that $\alpha\sim\alpha'$ and $\beta\sim\beta'$. Then by Lemma \ref{c2l2413}, we have that $\alpha\cdot\beta\sim\alpha'\cdot\beta'$. Hence the product $[\alpha]\cdot[\beta]=[\alpha\cdot\beta]$ is well defined. 
		
		Let $[\alpha], [\beta]\in\pi_1(Y,y)$. Then $[\alpha]\cdot[\beta]=[\alpha\cdot\beta]\in \pi_1(Y,y)$ since by definition of composition, $\alpha\cdot\beta$ is a loop in $Y.$ 
		
		The identity element is $[c_y]$ by Lemma \ref{c2l2411}. 
		Also, for any $[\alpha]\in\pi_1(Y,y)$, the inverse of $[\alpha]$ is $[\bar{\alpha}]$ because $\alpha\cdot\bar{\alpha}\sim c_y$ and $\bar{\alpha}\cdot\alpha\sim c_y$ by Lemma \ref{c2l2415}.
		
		It is only left for us to show associativity. From Lemma \ref{c2e2415}, we have that $\alpha\cdot(\beta\cdot\gamma)\sim (\alpha\cdot\beta)\cdot\gamma$. This implies that $[\alpha\cdot(\beta\cdot\gamma)]=[(\alpha\cdot\beta)\cdot\gamma].$
		Hence 
		$$([\alpha]\cdot[\beta])\cdot[\gamma]=([\alpha\cdot\beta])\cdot[\gamma]=[(\alpha\cdot\beta)\cdot\gamma]=[\alpha\cdot(\beta\cdot\gamma)]=[\alpha]\cdot([\beta\cdot\gamma])=[\alpha]\cdot([\beta]\cdot[\gamma]).$$
		Therefore the group operation is associative. Hence, $\pi_1(Y,y)$ is a group.
	\end{proof}
	The group $\pi_1(Y,y)$ is called the fundamental group of $Y$ relative to the base point $y$.
	\begin{pro}\cite{allen, paulst}:
		Let $X$ and $Y$ be topological spaces, and let $(x,y)\in (X,Y)$. Then $$\pi_1(X\times Y)\cong \pi_1(X)\times \pi_1(Y).\label{c2p2417}$$
	\end{pro}
	Now, we will look at the dependence of $\pi_1(X,x)$ on the choice of the basepoint $x.$ We suppose that $Y$ is a topological space, and $y_1,y_2\in Y$. Let $\gamma:[0,1]\rightarrow Y$ be a path from $y_1$ to $y_2.$ Define $\psi_\gamma : \pi_1(Y,y_1)\to \pi_1(Y,y_2)$ by
	\begin{eqnarray*}
		\psi_\gamma(\left[\alpha\right])=\left[\bar{\gamma}\cdot\alpha\cdot\gamma\right].
	\end{eqnarray*}
	\begin{pro}\cite{allen, paulst}:
		The map $\psi_\gamma$ defined above is an isomorphism.
	\end{pro}
	So we conclude that if $Y$ is a \textbf{path-connected} space (that is, for any $y_1,y_2\in Y$, there exists a path from $y_1$ to $y_2$), the group $\pi_1(Y,y)$ is up to isomorphism, independent of whatever choice of basepoint $y$. In this case, $\pi_1(Y,y)$ is often written as $\pi_1(Y).$ 
	\begin{thm}\cite{allen}:
		The fundamental group of a circle is $\mathbb{Z}$, that is, $\pi_1(S^1)=\mathbb{Z}.$
	\end{thm}\vspace{-2pt}
	There are different ways of finding the fundamental group of a topological space. For example if $X\cong Y$, then $\pi_1(X)\cong\pi_1(Y)$. This is what we will be discussing for the rest of this chapter. We begin with the induced homomorphism and deformation retraction, and then we state van Kampen's theorem and give its application.
	\section{Induced Homomorphism and Deformation Retract}\label{c2s5}
	A continuous map from a topological space into one of its subspaces where the position of all points of the subspace is preserved is called a retraction. In what follows, we give a formal definition of  a deformation retraction. We begin with the notion of the induced homomorphism.
	\begin{defn}[Induced homomorphism]
		Let $f:(X,x)\rightarrow (Y,y)$. Then the map
		\begin{eqnarray*}
		 f_{\ast}=\pi_1(f)=:\pi_1(X,x)&\to& \pi_1(Y,y)\\
		 \left[\alpha\right]&\mapsto&\left[f\circ \alpha\right].
		 \end{eqnarray*} is called the homomorphism induced by $f$.\label{s2d3}
	\end{defn}
	
	From this definition, we see that $f_\ast$ is a group homomorphism since for $[\alpha],[\beta]\in \pi_1(X)$, we have $f_\ast([\alpha]\cdot[\beta])=f_\ast([\alpha\cdot\beta])=[f\circ (\alpha\cdot\beta)]$. But then
	$$f\circ(\alpha\cdot\beta)(s)=f\circ\begin{cases}
		\alpha(2s),\text{ if } s\in[0,\frac{1}{2}]\\
		\beta(2s-1), \text{ if } s\in[\frac{1}{2},1]
	\end{cases}=\begin{cases}
	f\circ\alpha(2s),\text{ if } s\in[0,\frac{1}{2}]\\
	f\circ\beta(2s-1), \text{ if } s\in[\frac{1}{2},1]
	\end{cases}=(f\circ \alpha)\cdot(f\circ\beta)(s).$$
	Hence, $$f_\ast([\alpha]\cdot[\beta])=[f\circ(\alpha\cdot\beta)]=[(f\circ\alpha)\cdot(f\circ\beta)]=[f\circ\alpha]\cdot[f\circ\beta]=f_\ast([\alpha])\cdot f_\ast([\beta]).$$
	
	Also $\pi_1(\mathrm{Id}_X)=\mathrm{Id}_{\pi_1(X)}$ since for $[\alpha]\in \pi_1(X),~\pi_1(\mathrm{Id}_X)[\alpha]=[\mathrm{Id}_X\circ \alpha]=[\alpha]=\mathrm{Id}_{\pi_1(X)}[\alpha].$
	\begin{pro}
		Let $X$ and $Y$ be topological spaces. Suppose $X\cong Y$, then $\pi_1(X)\cong \pi_1(Y).$ \label{s2p6}
	\end{pro}
	\vspace{-10pt}
	\begin{proof}
		Suppose $X\cong Y$, then there exists a bijection $f:X\rightarrow Y$ such that $f$ and $g$ are continuous, where $g$ is the inverse of $f$. Let $f_\ast :\pi_1(X)\rightarrow \pi_1(Y)$ be the induced homomorphism as in Definition \ref{s2d3}. Also, let $g_\ast :\pi_1(Y)\rightarrow \pi_1(X)$ be the homomorphism induced by the map $g :Y\rightarrow X$. 
		
			Since  $f$ is a bijection, then $f\circ g=\mathrm{Id}_Y$, and $g\circ f=\mathrm{Id}_X.$ Therefore, 
			\begin{eqnarray*}
				g_\ast f_\ast&=&(g\circ f)_\ast=\pi_1(g\circ f)=\pi_1(\mathrm{Id}_X)=\mathrm{Id}_{\pi_1(X)},~\text{ and }\\
				f_\ast g_\ast&=&(f\circ g)_\ast = \pi_1(f\circ g)=\pi_1(\mathrm{Id}_Y)=\mathrm{Id}_{\pi(Y)}.
			\end{eqnarray*}
			Hence $f_\ast$ is a bijection, and therefore $f_\ast$ is an isomorphism. This then implies that  $\pi_1(X)\cong \pi_1(Y).$
	\end{proof}
	Even if $X$ is just homotopy equivalent to $Y$, we will have $\pi_1(X)\cong \pi_1(Y)$ as stated in Theorem \ref{c2t254} below
	\begin{pro}\cite{allen}
		If $f:X\to Y$ is a homotopy equivalence, then the induced map $\pi_1(f)=f_{\ast}:\pi_1(X,x)\to \pi_1(Y,f(x))$ is an isomorphism for all $x\in X.$\label{c2p252}
	\end{pro}
	\begin{thm}\cite{allen}
		Let $X$ and $Y$ be spaces. Assume that $X$ is homotopy equivalent to $Y$. Then $\pi_1(X)\cong \pi_1(Y)$. That is $X\simeq Y \implies \pi_1(X)\cong \pi_1(Y).$\label{c2t254}
	\end{thm}
	\begin{proof}
		This follows immediately from Proposition \ref{c2p252} above.
	\end{proof}
	\begin{exa}
		If $X$ is contractible, then $\pi_1(X)=0$ by the previous theorem. For instance, $\pi_1(\mathbb{R}^n)=0$ and $\pi_1(D^n)=0$ ($\mathbb{R}^n\simeq \mathrm{pt}$ and $D^n\simeq \mathrm{pt}$ by Example \ref{c2e1}).\label{c2e255}
	\end{exa}
	\begin{defn}[Deformation Retract]
		Let $X$ be a space, and let $A\subseteq X$. We say that $A$ is a retract of $X$ if there is a continuous map $r:X\to A$, called retraction, such that $r|_A=\mathrm{Id}_A$. If in addition $i\circ r\sim \mathrm{Id}_X$, then we say that $A$ is a deformation retract of $X$. (Here $i:A\to X$ is the inclusion map).
	\end{defn}\newpage
	Note that the concept of retract is not the same the same as the concept of deformation retract. By definition, any deformation retract is a retract. But the converse is not true. Namely, take $X=S^1$ and 	
	$A=\{(1,0)\}$. It is clear that $A$ is a retract of $X$. But by the following theorem, $A$ is not a deformation retract of $X$.
	\begin{thm}\cite{allen}
		Let $X$ be a topological space, and let $A\subseteq X.$ Assume that $X$ deformation retracts onto $A$. Then the inclusion $i:A\to X$ induces an isomorphism $i_{\ast}:\pi_1(A)\to \pi_1(X).$\label{c2t257}
	\end{thm}\vspace{-10pt}
	\begin{proof}
		Since $A$ is a deformation retract of $X$, we have that $r|_A=\mathrm{Id}_A$ and $i\circ r\sim \mathrm{Id}_X$, where $r:X\to A$ is the retraction. The equality $r|_A=\mathrm{Id}_A$ is the same as $r\circ i=\mathrm{Id}_A$. This implies that $r\circ i\sim \mathrm{Id}_A$ and $i\circ r\sim \mathrm{Id}_X$. Hence $i$ is a homotopy equivalence. Therefore, $i_{\ast}$ is an isomorphism by Proposition \ref{c2p252} above.
	\end{proof}
	\begin{defn}[Mapping cylinder] Let $f:X\to Y$ be a map. Then the \textbf{mapping cylinder}  of $f$, denoted by $M_f$, is the quotient space of the disjoint union $(X\times I)\amalg Y$ under the identification $(x,0)\sim f(x), \forall x\in X.$\label{c2dmc}		
	\end{defn}\vspace{-5pt}
	This means that the mapping cylinder of  a map $f:X\to Y$, is obtained by gluing $X\times \{0\}$ to $Y$ via $f$.
	\begin{pro}
		Let $f:X\to Y$ be a continuous map. Then the mapping cylinder of $f$ deformation retracts onto $Y$.\label{c2pm}
	\end{pro}
	\begin{proof}
		By definition of $M_f$, we see that $Y$ is a subspace of $M_f$. Define the map $r:M_f\to Y$ by $r([(x,t)])=f(x),~ x\in X, t\in I$ and $r(y)=y,~ y\in Y$. Then $r$ is continuous since $f$ is continuous and $y$ is a polynomial. Moreover
		$$r|_Y(y)=y,~\forall y\in Y.$$
		Hence, $Y$ is a retract of $M_f$. Now, to show that $Y$ is a deformation retract of $M_f$, we  need to show that $i\circ r\sim \mathrm{Id}_{Mf}$, where $i:Y\to M_f$ is the inclusion map. So we need to find a homotopy $F:M_f\times I\to M_f$ such that
		$$F(a,0)=i\circ r(a),\text{ and } F(a,1)=a\text{ for all }a\in M_f.$$ Define $F([(x,t)],s)=[(x,t\cdot s)]$ (The point of $Y$ stay fixed during the homotopy). This is well-defined since $[(x,0)]=[(x,0\cdot s)]$ for all $s$. Now let $a=[(x,t)]\in M_f$ be arbitrary, then $$F(a,0)=F([(x,t)],0)=[(x,0)].$$ But $$i\circ r(a)=i\circ r([(x,t)])=i(f(x))=[f(x)]=[(x,0)]\text{ (since $(x,0)\sim f(x)$ from Definition \ref{c2dmc})}.$$
		Thus $F(a,0)=i\circ r(a).$ Also
		$$F(a,1)=F([(x,t)],1)=[(x,t)]=a.$$
		Hence $F$ is the desired homotopy. Thus, the mapping cylinder of $f$ deformation retracts onto $Y$.
	\end{proof}
	\section{Van Kampen's Theorem}\label{c2s6}
	Computing fundamental group of a very large space may be difficult. But then, van Kampen's theorem gives a way of computing fundamental group of a space (large or not) which can be decomposed into 
	simpler spaces whose fundamental groups are known already. We recall from Section \ref{c2s3} that the free product notation is $``\star"$.
	\newpage
	\begin{thm}[van Kampen's Theorem \cite{allen}]
		Let $X$ be a topological space, and let $U, V$ be open subsets of $X$ such that $X=U\cup V$. Assume that the basepoint $x_0\in U\cap V$, and $U,V,U\cap V$ are path connected. Consider the following diagrams:
		\begin{center}
			\includegraphics[height=3cm]{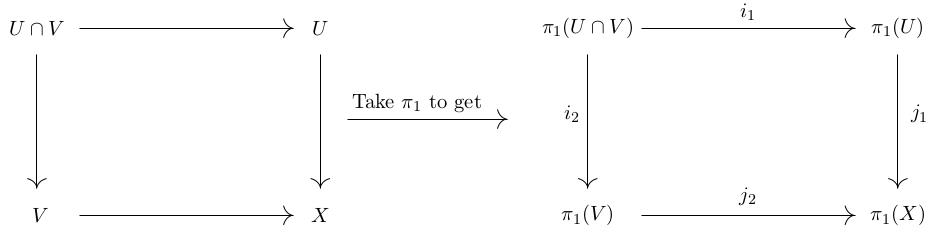}
		\end{center}
	Then
	\begin{enumerate}
		\item the canonical map $\phi:\pi_1(U)\star \pi_1(V)\rightarrow \pi_1(X)$ is surjective, and
		\item $\ker(\phi)$ is the normal subgroup $N$ generated by $i_1(g)i_2(g)^{-1}\in\pi_1(U)\star\pi_1(V),~g\in \pi_1(U\cap V)$.
	\end{enumerate}
	Hence, by the first isomorphism theorem, we conclude that $\pi_1(X)\cong \dfrac{\pi_1(U)\star\pi_1(V)}{N}$. 
	\end{thm}
	The proof of this theorem can be found on page 44 of the book of Hatcher \cite{allen}. In what follows, we use van Kampen's theorem to compute the fundamental group of wedge of two circles.
	\begin{exa}[Wedge of two circles]
		We assume that $X$ and $Y$ are two circles . Let $x_0$ and $y_0$ be the basepoints of $X$ and $Y$ respectively. Then, the wedge of $X$ and $Y$, written as $X\vee Y$, is obtained by identifying these two basepoints. This scenario is illustrated in the diagram below
		\begin{center}
			\includegraphics[height=3cm]{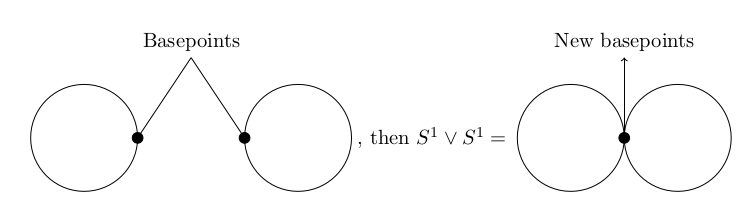}
		\end{center}
	Now, we will use Van Kampen's theorem to compute its fundamental group. We let
	\begin{center}
		\includegraphics[height=2cm]{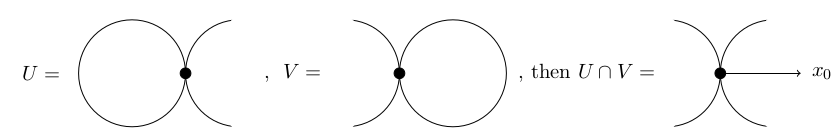}
	\end{center}
	So, $U$ and $V$ are open and path-connected, $x_0\in U\cap V, ~U\cap V$ is path-connected, and $S^1\vee S^1=U\cup V.$
	We see that $U\simeq S^1,\text{ hence }\pi_1(U)\cong \pi_1(S^1)=\mathbb{Z}.\text{ Also } V\simeq S^1,\text{ hence }\pi_1(V)\cong \mathbb{Z}$. Since $U\cap V\simeq \{x_0\}$, it follows that $\pi_1(U\cap V)\cong 0.$	
	Hence Van Kampen's theorem says that
	\begin{equation}\pi_1(S^1\vee S^1)\cong\frac{\pi_1(U)\star \pi_1(V)}{N}.\label{c2e261}\end{equation}
	If $g\in \pi_1(U\cap V)$, then $g=0$. Hence $i_1(g)=e_1$. Similarly, $i_2(g)=e_2.$ So $i_1(g)i_2(g)^{-1}=e_1e_2=e$, the 
	\newpage
	empty word in $\pi_1(U)\star \pi_1(V).$ But $N=\langle i_1(g)i_2(g)^{-1}\rangle=\langle e \rangle$. Thus $N\cong 0.$ Hence $\eqref{c2e261}$ implies that $$\pi_1(S^1\vee S^1)\cong \pi_1(U)\star \pi_1(V)\cong \mathbb{Z}\star \mathbb{Z}.$$
	\end{exa}
	\begin{exa}
		$\pi_1(S^n)\cong0,~n\ge 2.$ We suppose that $U=S^n\setminus\{(0,\dots,1)\}$, and $V=S^n\setminus\{0,\dots,-1\}$. Hence $U\cap V=S^n\setminus\{(0,\dots,1), (0,\dots,-1)\}$. By this decomposition, we see that $U$, $V$, and $U\cap V$ are open and path-connected. We also have that $S^n=U\cup V$. Let $x_0\in U\cap V$, then by van Kampen's theorem, we have that 
		$$\pi_1(S^n)\cong \frac{\pi_1(U)\star\pi_1(V)}{N}.$$
		But by Proposition \ref{s2p5}, we have that $U\cong \mathbb{R}^3$. Hence $\pi_1(U)\cong \pi_1(\mathbb{R}^3)=0.$ Similarly, $\pi_1(V)=0$.
		Therefore $$\pi_1(S^n)\cong 0.\label{c2e263}$$
	\end{exa}
	
	Having studied these concepts in group theory and fundamental groups, we will now compute and describe the fundamental group of torus knots in the next chapter.  
\chapter{The Fundamental Group of Torus Knots}
\label{chap3}
The knot group of any knot is the fundamental group of the complement of the knot. So, to find the fundamental group of torus knots, it is as good as finding the fundamental group of the knots' complements. The goal of this chapter is to compute the knot group of torus knots, to describe the structure of the group, and finally to compute the knot group of an arbitrary knot. 
Throughout this chapter, we will use the standard notations $``\simeq"$, and $``\cong"$ for ``homotopy equivalent to" and ``homeomorphic to" respectively. We start with the following definitions.
\begin{defn}[Knot] 
	A knot is an embedding $f:S^1\to \mathbb{R}^3$ of the unit circle inside $\mathbb{R}^3$
\end{defn}
We recall that a map $f:X\to Y$ is an embedding if $f:X\to f(X)$ is a homeomorphism. The map $f:X\to Y$ needs not be a homeomorphism. But the map from $X$ to $f(X)$ that sends $x$ to $f(x)$ is a homeomorphism. If $f:X\to Y$ is an embedding, then in particular $f:X\to Y$ is injective.

Consider the standard embedding of the torus $S^1\times S^1$ in $\mathbb{R}^3$. So $S^1\times S^1\subseteq \mathbb{R}^3$. Also, consider the map $f:S^1\to S^1\times S^1$ defined by $f(z)=(z^m,z^n).$
\begin{pro}\cite{chris}:
	The map $f:S^1\to S^1\times S^1$ defined by $f(z)=(z^m,z^n)$ is an embedding if and only if $\gcd(m,n)=1.$\label{c3p302}
\end{pro}
\begin{proof}
	The map $f$ is continuous since its components are polynomial functions. Since $S^1$ is compact and $S^1\times S^1$ is Hausdorff, then $f$ is an embedding if it is injective. So we will prove that $f$ is injective if and only if $\gcd(m,n)=1$. 
	
	Suppose that $\gcd(m,n)=1$.  We want to show that $f$ is injective.  Let $z,w\in S^1$, and let $\theta, \theta'\in [0,1)$ such that $z=e^{2\pi i\theta}$, and $w=e^{2\pi i \theta'}$. Suppose that $f(z)=f(w)$, then
	$$(z^m,z^n)=(w^m,w^n)\implies \left(e^{2\pi i\theta m},e^{2\pi i\theta n}\right)=\left(e^{2\pi i\theta' m},e^{2\pi i\theta' n}\right).$$
	This implies that $e^{2\pi i\theta m}=e^{2\pi i\theta' m}$ and $e^{2\pi i\theta n}=e^{2\pi i\theta' n}$. Hence we have that
	$$\frac{e^{2\pi i\theta m}}{e^{2\pi i\theta' m}}=1\implies e^{2\pi i m(\theta-\theta')}=1\implies m(\theta-\theta')\in\mathbb{Z}\text{ (Assuming $\theta\ge \theta'$)}.$$
	Similarly, $n(\theta-\theta')\in\mathbb{Z}$. It then follows that $\theta-\theta'=\frac{p}{q}$, where $p$ and $q$ can be chosen such that $\gcd(p,q)=1$. Hence $q|m$ and $q|n$. But $\gcd(m,n)=1$. Hence $q=1$, and thus $\theta-\theta'=p\in\mathbb{Z}$. By our assumption, $\theta-\theta'\ge 0$, and $\theta, \theta'\in[0,1)$. Hence $\theta-\theta'=0\implies \theta=\theta'$. Therefore
	$$z=e^{2\pi i\theta}=e^{2\pi i\theta'}=w.$$
	Hence $f$ is injective.
	
	Conversely, suppose that $f$ is injective. We want to show that $\gcd(m,n)=1$. Suppose by contradiction that $\gcd(m,n)=d>1$. Then $\frac{1}{d}\in (0,1)$, and $\frac{m}{d},\frac{n}{d}\in\mathbb{Z}$. Thus
	$$\left(e^{2m\pi i}, e^{2n\pi i}\right)=\left(e^{(2m\pi i)/d}, e^{(2n\pi i)/d}\right).$$
	This implies that $f$ is not injective. Hence, a contradiction. So we must have that $\gcd(m,n)=1.$
\end{proof}
From now on, we will assume that $\gcd(m,n)=1.$
\begin{defn}[Torus knots]
	The map $f:S^1\to S^1\times S^1$, defined by $f(z)=(z^m,z^n)$ is a knot by Proposition \ref{c3p302}. This knot is called torus knot. We identify $f$ with its image, say $K$. So $K:=\mathrm{Im}(f).$
\end{defn}

\section{Calculating the Fundamental Group of $\mathbb{R}^3\setminus K$}
It is slightly easier to compute $\pi_1(\mathbb{R}^3\setminus K)$ if we replace $\mathbb{R}^3$ by its one-point compactification, $S^3$. So to compute $\pi_1(\mathbb{R}^3\setminus K)$, we will first show that $\pi_1(\mathbb{R}^3\setminus K)\cong \pi_1(S^3\setminus K)$, then we will show that $S^3\setminus K$ deformation retracts onto a 2-dimensional complex $X=X_{m,n}$ which is homeomorphic to the quotient space of a cylinder $S^1\times I$ under the identifications $(z,0)\sim (e^{\frac{2\pi i}{m}}z,0)$ and $(z,1)\sim (e^{\frac{2\pi i}{n}}z,1)$, and lastly we will compute $\pi_1(X)$. We can then conclude that  $\pi_1(\mathbb{R}^3\setminus K)\cong \pi_1(S^3\setminus K)\cong \pi_1(X).$

\begin{pro}
	The fundamental groups of $\mathbb{R}^3\setminus K$ and   ${S^3\setminus K}$ are isomorphic.
\end{pro}
\begin{proof}
Let $\phi: S^3\setminus p\to \mathbb{R}^3$ be the stereographic projection as in Proposition \ref{s2p5}, where $p$ is the north pole. We first show that $\phi^{-1}(K)\cong K.$ To show this, define $f:K\to \phi^{-1}(K)$ by $f(z)=\phi^{-1}(z)$, and $f^{-1}:\phi^{-1}(K)\to K$ by $f^{-1}(z)=\phi(z).$ The maps $f$ and $f^{-1}$ are continuous since $\phi$ is a homeomorphism as stated in proposition \ref{s2p5}. And since
$$f\circ f^{-1}(z)=f(\phi(z))=\phi^{-1}(\phi(z))=z=\mathrm{Id}_{\phi^{-1}(K)},\text{ }$$
$$f^{-1}\circ f(z)=f^{-1}(\phi^{-1}(z))=\phi(\phi^{-1}(z))=z=\mathrm{Id}_K,$$
then we conclude that $f$ is a bijection. Thus $\phi^{-1}(K)\cong K.$ Also, by taking $n=3$ in Proposition \ref{s2p5}, one has $S^3\setminus \{p\}\cong \mathbb{R}^3.$

Next, we will use Van Kampen's theorem to show that $\pi_1(S^3\setminus K)\cong \pi_1(\mathbb{R}^3\setminus K)$. To do this, we suppose that $B_1$ is a large closed ball in $\mathbb{R}^3$ which contains $K$. Define $B=\phi^{-1}(\mathbb{R}^3\setminus B_1)\cup \{p\}$, then
\begin{equation}
S^3\setminus \phi^{-1}(K)=\left((S^3\setminus \{p\})\setminus \phi^{-1}(K)\right)\cup B. \label{s3e1}
\end{equation} 
But then, $S^3\setminus \{p\}\cong \mathbb{R}^3,$ and $\phi^{-1}(K)\cong K$, hence \eqref{s3e1} becomes
\begin{equation}
S^3\setminus K\cong (\mathbb{R}^3\setminus K)\cup B.\label{s3e2}
\end{equation} 
The space $B=\phi^{-1}(\mathbb{R}^3\setminus B_1)\cup \{p\}\cong (\mathbb{R}^3\setminus B_1)\cup \{p\}$ since $\phi^{-1}(\mathbb{R}^3\setminus B_1)\cong \mathbb{R}^3\setminus B_1$. Thus, ($\mathbb{R}^3\setminus K)\cap B\cong S^2\times \mathbb{R}$ as illustrated in Figure \ref{fig3.1}.
\begin{figure}[htbp!]
	\centering
	\includegraphics[width=0.8\textwidth]{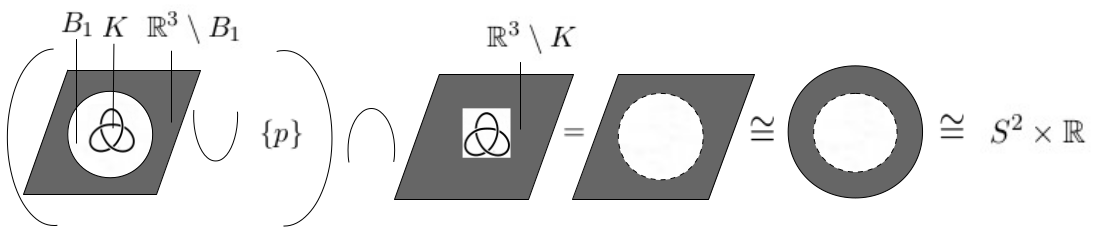}
	\caption{$B\cap (\mathbb{R}^3\setminus K).$}\label{fig3.1}
\end{figure}
 Certainly,  $\mathbb{R}^3\setminus K, B$ and $S^2\times \mathbb{R}$ are open and  
path-connected. Let $x_0\in S^2\times \mathbb{R}$ be the basepoint. Then 
by van Kampen's theorem, we have that
\begin{equation}
\pi_1(S^3\setminus K)\cong \dfrac{\pi_1(\mathbb{R}^3\setminus K)\star \pi_1(B)}{N},\label{s3e3}
\end{equation}
where $N$ is the normal subgroup generated by $i_1(g)i_2(g)^{-1}\in\pi_1(\mathbb{R}^3\setminus K)\star\pi_1(B),~g\in \pi_1(S^2\times \mathbb{R})$.  

But $\pi_1(S^2\times \mathbb{R})\cong \pi_1(S^2)\times \pi_1(\mathbb{R})$ by Proposition \ref{c2p2417}. Also by Examples \ref{c2e255} and \ref{c2e263}, we have that $\pi_1(S^2)\cong 0$ and $\pi_1(\mathbb{R})\cong 0$. Therefore $\pi_1(S^2\times\mathbb{R})\cong \pi_1(S^2)\times\pi_1(R)\cong 0\times 0\cong 0$. Hence if $g\in\pi_1(S^2\times \mathbb{R})$, then $g=0.$ Thus $i_1(g)i_2(g)^{-1}=e,$ the empty word in
 $\pi_1(\mathbb{R}^3\setminus K)\star \pi_1(B).$ Therefore $N=\langle i_1(g)i_2(g)^{-1}\rangle\cong \langle e\rangle =0$. Also, since $B$ is a ball, then $\pi_1(B)\cong 0.$ Hence \eqref{s3e3} becomes 
\begin{equation}
\pi_1(S^3\setminus K)\cong \pi_1(\mathbb{R}^3\setminus K).\label{s3e4}
\end{equation}
\end{proof}
\vspace{-10pt}
\subsection{ Deformation Retraction of $\boldsymbol{{\rm I\!R}^3\setminus K}$ onto a 2-dimensional Complex, \boldmath{$X$}}
The next thing we will do is to show that $S^3\setminus K$ deformation retracts onto a 2-dimensional complex $X=X_{m,n}$ which is homeomorphic to the quotient space of a cylinder $S^1\times I$ under the identifications $(z,0)\sim (e^{\frac{2\pi i}{m}}z,0)$ and $(z,1)\sim (e^{\frac{2\pi i}{n}}z,1)$. This will allow us to conclude that $\pi_1(S^3\setminus K)\cong \pi_1(X)$ as shown in Theorem \ref{c2t257}.

	First of all, we need to define the space $X=X_{m,n}.$ But before then, we will define the space $X_m$ and $X_n$ as follows. 	
	We regard $S^3$ as $S^3\cong (S^1\times D^2)\cup(D^2\times S^1)$. Let us identify the first solid torus, $S^1\times D^2$, with the compact region of $\mathbb{R}^3$ bounded by the standard torus $S^1\times S^1$ containing $K$. The second torus $D^2\times S^1$ is then the closure of the complement of the first solid torus, together with the compactification point at infinity.
	
	Consider the first solid torus, and let $x\in S^1.$ Then, one has the meridian disk $\{x\}\times D^2$. Since $f:S^1\to S^1\times S^1$ (this is the map that defines $K$) is an embedding, and since $K$ winds around the torus a total of $m$ times in the longitudinal direction, the knot $K$ intersects the meridian circle $\{x\}\times \partial D^2$ in $m$ equally spaced points, say $P_1, P_2,\dots, P_m$. For $i\in\{1,\dots, m-1\}$, let $P_{i(i+1)}$ be the mid-point of the arc $P_iP_{i+1}$, and let $P_{m1}$ be the mid-point of the arc $P_mP_1$. See Figure~\ref{fig:m} for the case when $m=3$.
	\begin{figure}[htbp!]
		\centering
		\includegraphics[width=0.25\textwidth]{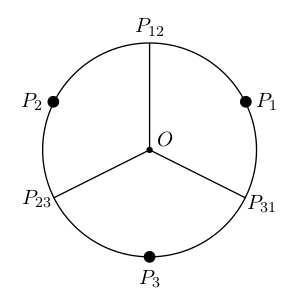}
		\caption{The case when $m=3.$}\label{fig:m}
	\end{figure}
	Consider the line segments $OP_{i(i+1)}, ~1\le i\le m-1$, and $OP_{m1}.$ Define $S_x$ by $S_x=\left(\underset{i=1}{\cup}OP_{i(i+1)}\right)\cup OP_{m1}.$
	In other words, $S_x$ is the union of all those line segments. Now, define
	$X_m=\underset{x\in S^1}{\cup}S_x.$
	The space $X_m$ is the space traced out by the line segments when $x$ runs over $S^1$. Similarly, by considering the second solid torus $D^2\times S^1$, one defines $X_n$ as $X_n=\underset{z\in S^1}{\cup}T_z.$ Now, we define the space $X$ as 
	\begin{equation}
		X=X_m\cup X_n.\label{c3e315}
	\end{equation}
	By this definition of $X$, we can see that $X$ is a subspace of $S^3\setminus K.$ Again by definition of $X_m$ and $X_n$, one can see that $X_m$ is the mapping cylinder of the map $\alpha_m:S^1\to S^1$ by $\alpha_m(z)=z^m$. And $X_n$ is the mapping cylinder of the map $\alpha_n:S^1\to S^1$ by $\alpha_n(z)-z^n.$ Now, we will show the following result.
	\begin{pro}
		(a) The space $X_m$ we just defined is homeomorphic to the quotient space of the cylinder $S^1\times [0,\frac{1}{2}]$, under the identification $(z,0)\sim (e^{\frac{2\pi i}{m}}z,0)$.
		
		(b) The space $X_n$ we just defined is homeomorphic to the quotient space of the cylinder $S^1\times [\frac{1}{2},1]$, under the identification $(z,1)\sim (e^{\frac{2\pi i}{n}}z,1).$\label{c3p313} 
	\end{pro}
	\begin{proof}
		Let $Y$ be the quotient space of $S^1\times [0,\frac{1}{2}]$, under the identification $(z,0)\sim (e^{\frac{2\pi i}{m}}z,0)$. So $Y=\dfrac{S^1\times [0,\frac{1}{2}]}{\sim}$. Since $X_m$ is the mapping cylinder of $\alpha_m:S^1\to S^1, z\mapsto z^m,$ we have 
		$$X_m=\dfrac{(S^1\times [0,1])\amalg S^1}{\sim'},$$
		where $\sim'$ is the relation: $(z,0)\sim' \alpha_m(z)=z^m.$
		
		We will define two maps $\overline{\varphi}: X_m\to Y$ and $\overline{\psi}:Y\to X_m$ such that $\overline{\psi}\overline{\varphi}=\mathrm{Id}_{X_m}$ and $\overline{\varphi}\overline{\psi}=\mathrm{Id}_Y.$
		
		$\bullet$ \underline{Defining $\overline{\varphi}$} : First define $\varphi:(S^1\times [0,1])\amalg S^1\to Y$ by  $$\varphi(z,t)=[(z,t/2)], 0\le t\le 1, \text{ and } \varphi(z)=[(e^{\frac{i\theta}{m}},0)]\text{ if }z=e^{i\theta}.$$
		Clearly, if $(z,0)\sim' z^m,$ then $\varphi(z,0)=\varphi(z^m)$. So $\varphi$ passes to the quotient and gives rise to a map
		\begin{eqnarray*}
		\overline{\varphi}:\frac{(S^1\times [0,1])\amalg S^1}{\sim'}&\to& Y,\text{ defined by }\overline{\varphi}([x])=\left[\varphi(x)\right].
		\end{eqnarray*}
		By this definition, $\overline{\varphi}$ is continuous. 
		
		$\bullet$ \underline{Defining $\overline{\psi}$} : First define $\psi:S^1\times[0,1]\to X_m$ by 
		$$\psi(z,0)=[z^m]=[(z,0)],\text{ and }\psi(z,t)=[(z,2t)],~ 0<t\le \frac{1}{2}.$$
		If $(z,0)\sim (e^{\frac{2\pi i}{m}}z,0)$, then $\psi(z,0)=[(z,0)]=\psi(e^{\frac{2\pi i}{m}}z,0)$. Passing to the quotient, we get 
		$$\overline{\psi}: \frac{S^1\times [0,\frac{1}{2}]}{\sim}\to X_m,\text{ defined by } \overline{\psi}([x])=[\psi(x)].$$
		Again, by definition, $\overline{\psi}$ is continuous. Now, we will show that $\overline{\psi}\overline{\varphi}=\mathrm{Id}_{X_m}$ and $\overline{\varphi}\overline{\psi}=\mathrm{Id}_{Y}$. For the case when $t\ne 0$, we suppose that $[z,t]\in X_m$, and $[x,t]\in Y$, we then have that
		$$\overline{\psi}\overline{\varphi}([z,t])=\overline{\psi}([\varphi(z,t)])=\overline{\psi}([z,t/2])=[\psi(z,t/2])]=[z,t]=\mathrm{Id}_{X_m},\text{ and }$$
		$$\overline{\varphi}\overline{\psi}([x,t])=\overline{\varphi}([\psi(x,t)])=\overline{\varphi}([(x,2t)])=[\overline{\varphi}(x,2t)]=[x,t]=\mathrm{Id}_Y.$$
		For the case when $t=0$, we have that
		$$\overline{\psi}\overline{\varphi}([z,0])=\overline{\psi}([\varphi(z,0)])=\overline{\psi}([z,0])=[\psi(z,0)]=[z^m]=[z,0]=\mathrm{Id}_{X_m}([z,0]),\text{ and }$$
		$$\overline{\varphi}\overline{\psi}([x,0])=\overline{\varphi}([\psi(x,0)])=\overline{\varphi}([x^m])=\overline{\varphi}([x,0])=[\varphi(x,0)]=[x,0]=\mathrm{Id}_{Y}.$$
		Therefore $X_m\cong Y.$
		
		Part (b) can be handled in the same way.
	\end{proof}
	\begin{pro}
		Let $X$ be as defined in \eqref{c3e315}. Then $X$ is homeomorphic to the quotient space of $S^1\times I$ under the identifications $(z,0)\sim (e^{\frac{2\pi i}{m}}z,0)$ and $(z,1)\sim (e^{\frac{2\pi i}{n}}z,1)$.\label{c3p314}
	\end{pro}
	\begin{proof}
		Since $X=X_m\cup X_n$ from \eqref{c3e315}. Then by Proposition \eqref{c3p313} above, the result follows.
	\end{proof}
	\newpage
	Thanks to Proposition \ref{c3p314} above, $X_m\cap X_n=S^1\times\{\frac{1}{2}\}$.
	Finally, we will show that $S^3\setminus K$ deformation retracts onto $X$. The idea is 
	\begin{itemize}
		\item to construct a deformation retraction of the first solid torus $S^1\times D^2$ minus $K$ onto $X_m$, and 
		\item to construct a deformation retraction of the second solid torus $D^2\times S^1$ minus $K$ onto $X_n$
	\end{itemize}
	such that they both agree on the intersection $S^1\times S^1\setminus K.$
	
	For the deformation retraction of $(S^1\times D^2)\setminus K$ onto $X_m$, we will construct a deformation retraction on each meridian disk. Let $x\in S^1$. Need to define a continuous map	 
	$(\{x\}\times D^2)\setminus K\to S_x$. To do this, it suffices to show that each region (for example, when $m=3$, we have 3 regions, each contains only one point $P_i$) deformation retracts onto the radial line segments that bound it.
	
	So we have $m$ regions namely, $R_1, R_2, \dots ,R_m$, with $P_i\in R_i.$ We will show that $R_i\setminus \{P_i\}$ deformation retracts onto the space $OP_{i(i+1)}\cup OP_{(i-1)i}$ in a nice way. This will follow from the following Lemma.
	
	\begin{lem}
		Consider the region (or space) ${R}$ in $\mathbb{R}^2$ in Figure \ref{fig3.3} below. Then the space ${Y=R\setminus\{(0,1)\}}$ deformation retracts onto ${A=OP\cup OQ.}$ \label{c3l315}
		\vspace{-5pt}
		\begin{figure}[htbp!]
			\begin{center}
				\begin{minipage}[b]{0.3\linewidth}
					\centering
					\includegraphics[width=1.\linewidth]{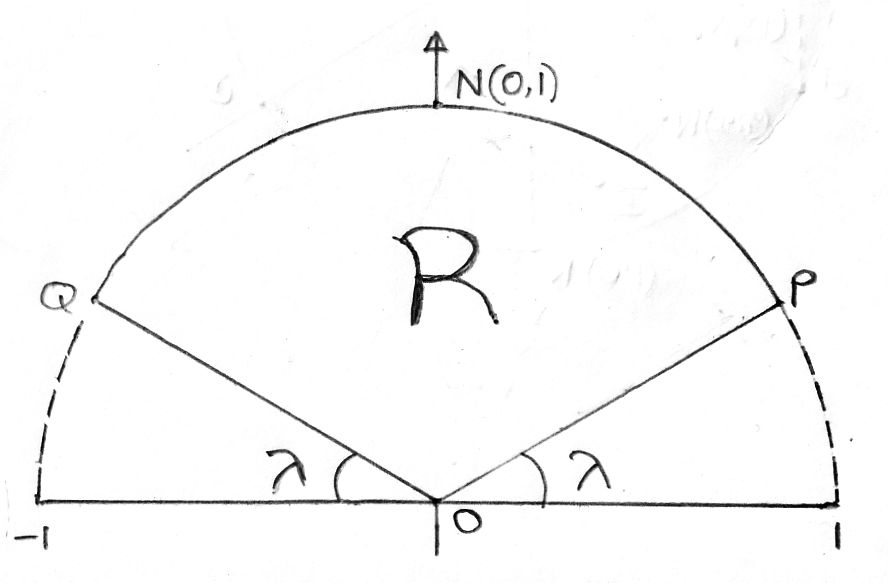} 
					\caption{Region $R$}\label{fig3.3} 
					\vspace{2ex}
				\end{minipage}%%
				\begin{minipage}[b]{0.3\linewidth}
					\centering
					\includegraphics[width=1.\linewidth]{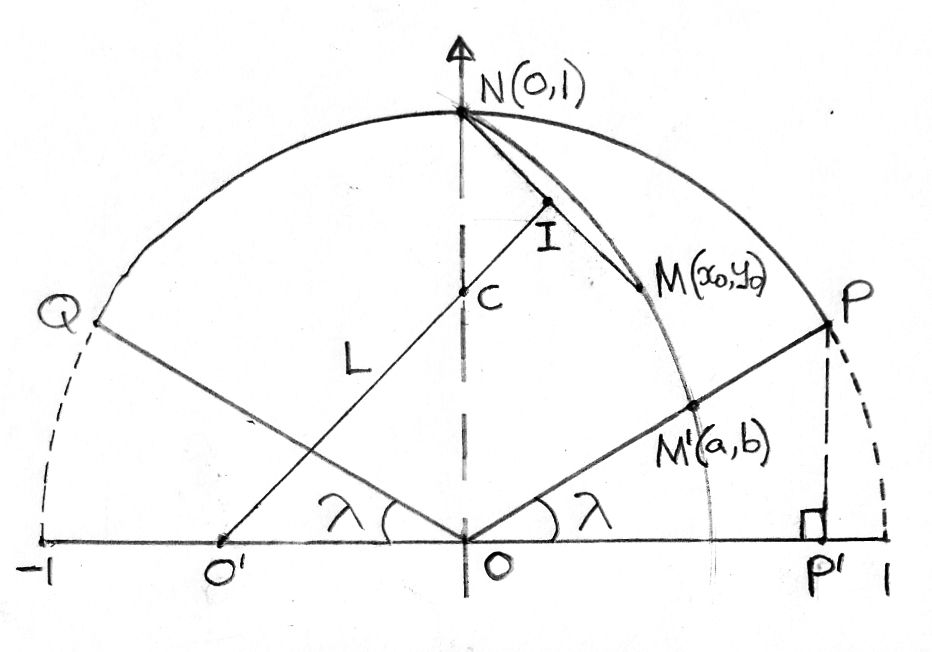} 
					\caption{When $x_0>0$}\label{fig3.4}
					\vspace{1.1ex}
				\end{minipage}
				\begin{minipage}[b]{0.3\linewidth}
					\centering
					\includegraphics[width=1.\linewidth]{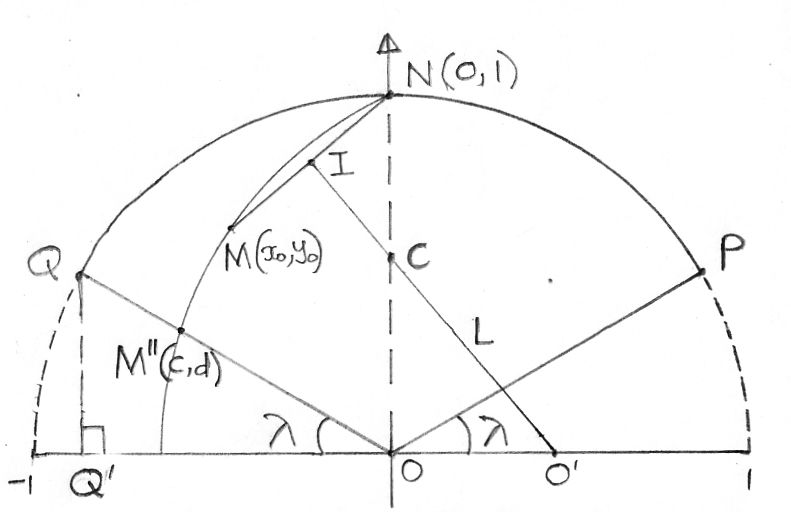} 
					\caption{When $x_0<0$}\label{fig3.5}
					\vspace{2.3ex}  
				\end{minipage}
			\end{center}   
		\end{figure}
	\end{lem}	
	\vspace{-35pt}
	\begin{proof}
	To prove this, we define a retraction $r:Y\to A$ using suitable flows. This will allow us to distort the flows (if necessary) later in such a way that the deformation retracts agree on the intersection of the two solid tori, $S^1\times S^1\setminus K.$
	
	To construct the map $r$, let $M(x_0,y_0)\in Y$. We consider the following three cases: 
	
	$\bullet$ {If $x_0>0$}.
		\begin{enumerate}
			\item Consider the midpoint $I$ of the line segment $MN$.
			\item Draw the line, say $L$, through $I$ and orthogonal to $MN$. Let $O'$ be the intersection point between $L$ and the $x$-axis.
			\item Draw the circle, say $C$, with center $O'$ and radius $O'N$. That is, $C$ is the circle through $N$ with center $O'.$
			\item Let $M'$ be the intersection between $C$ and the line segment $OP$. Define $r(M)=M'.$  
		\end{enumerate}
	$\bullet$ {If $x_0<0$}, one defines $r(M)=M''$, where $M''$ is the intersection between the circle and $OQ.$
	
	$\bullet$ {If $x_0=0$}, define $r(M)=0.$ 
	 
	Next, we will write down the map $r:Y\to A$ explicitly, and then show that it is a deformation retraction 
	of $Y$ onto $A$. We define the map for each case. For the first case, that is, if $x_0>0,$ see Figure \ref{fig3.4} for the construction.
	\newpage
	Since the line $L$ is orthogonal to $MN$, then $S_L=-\dfrac{1}{S_{MN}}$, where $S_L$ is the slope of line $L$, and $S_{MN}$ is the slope of line $MN$. Since $M=(x_0,y_0)$ and $N=(0,1)$, then 
	$$S_{MN}=\frac{1-y_0}{-x_0}=\frac{y_0-1}{x_0}.$$
	Hence, $$S_L=-\frac{x_0}{y_0-1}=\frac{x_0}{1-y_0}.$$
	Since $I$ is the midpoint of $MN$, then the coordinates of $I$ is $\left(\dfrac{x_0}{2},\dfrac{1+y_0}{2}\right)$. Thus, the equation of $L$ is 
	\begin{equation}
	y=\frac{x_0}{1-y_0}x+C,\label{s3e5}
	\end{equation}
	where $C$ is the $y$-intercept. To find $C$, we substitute the coordinates of $I$ into \ref{s3e5}, so that we have 
	$$\frac{1+y_0}{2}=\frac{x_0}{1-y_0}\cdot\frac{x_0}{2}+C\implies C=\frac{1-y_0^2-x_0^2}{2(1-y_0)}.$$
	Hence the equation of line $L$ is 
	\begin{equation}
	y=\frac{x_0}{1-y_0}x+\frac{1-y_0^2-x_0^2}{2(1-y_0)}.\label{s3e6}
	\end{equation}
	Next, we will find the coordinates of the $x$-intercept of line $L$, that is, the coordinates of the point $O'$ in Figure \ref{fig3.4}. At this point, $y=0$, hence we substitute $y=0$ in equation \eqref{s3e6} so that we have 
	$$x=\frac{x_0^2+y_0^2-1}{2x_0}.$$
	Hence the coordinates of the point $O'$ are $\left(\dfrac{x_0^2+y_0^2-1}{2x_0},0\right)$. Now consider the radius $O'N$ in Figure \ref{fig3.4}. This radius is the distance between the points $O'$ and $N$. Let this distance be $d_1$. Since $N=(0,1)$, then 
	\begin{equation}
	 d_1^2=1+\frac{(x_0^2+y_0^2-1)^2}{4x_0^2}=\frac{4x_0^2+(x_0^2+y_0^2-1)^2}{4x_0^2}.\label{s3e7}
	\end{equation}
	If we suppose that the point $M'=(a,b)$, then the distance between $O'$ and $M'$ is given by the relation
	\begin{eqnarray}
	d_2^2&=&\left(a-\frac{(x_0^2+y_0^2-1)}{2x_0}\right)^2+b^2\nonumber\\
	&=&\left(\frac{2ax_0-(x_0^2+y_0^2-1)}{2x_0}\right)^2+b^2\nonumber\\
	&=&\frac{[2ax_0-(x_0^2+y_0^2-1)]^2}{4x_0^2}+b^2\nonumber\\
	&=&\frac{4b^2x_0^2+[2ax_0-(x_0^2+y_0^2-1)]^2}{4x_0^2}.\label{s3e8}
	\end{eqnarray}
	But then $O'N=O'M'$ since they are both radii of the circle $C$. Hence, $d_1^2=d_2^2$. Therefore, we have
	$$
		\frac{4x_0^2+(x_0^2+y_0^2-1)^2}{4x_0^2}=\frac{4b^2x_0^2+[2ax_0-(x_0^2+y_0^2-1)]^2}{4x_0^2}.$$
	\begin{equation}\text{This then implies that: } \frac{4x_0^2+(x_0^2+y_0^2-1)^2}{4x_0^2}=\frac{4b^2x_0^2+4a^2x_0^2-4ax_0(x_0^2+y_0^2-1)+(x_0^2+y_0^2-1)^2}{4x_0^2}.\label{s3e9}
	\end{equation}
	\newpage
	By simplifying equation \eqref{s3e9}, we have that 
	\begin{equation}
	x_0^2=b^2x_0+a^2x_0-a(x_0^2+y_0^2-1)\label{s3e10}
	\end{equation}
	Now, consider the right angle triangle $OPP'$. Since $|OP|=1$, then $|OP'|=\cos(\lambda)$ and $|PP'|=\sin(\lambda)$. Hence the coordinates of the point $P$ are $(\cos(\lambda),\sin(\lambda)).$ Hence the slope of $OP$ is $\dfrac{\sin(\lambda)}{\cos(\lambda)}=\tan(\lambda)$. Also, the slope of $OM'$ is $\dfrac{b}{a}$. But then, slope of $OP$ is equal to slope of $OM'$. Hence 
	\begin{equation}
	\tan(\lambda)=\frac{b}{a}\implies b-a\tan(\lambda)=0.\label{s3e11}
	\end{equation}
	Solving equations \eqref{s3e10} and \eqref{s3e11} simultaneously, we have that 
	\begin{eqnarray*}
	a&=&\frac{x_0^2+y_0^2+\sqrt{x_0^4+y_0^4+2x_0^2(2t^2+1)+2y_0^2(x_0^2-1)+1}-1}{2x_0(t^2+1)},\text{ and }\\
	b&=&\frac{tx_0^2+ty_0^2+t\left(\sqrt{4t^2x_0^2+x_0^4+2x_0^2y_0^2+y_0^4+2x_0^2-2y_0^2+1}-1\right)}{2x_0(t^2+1)},
	\end{eqnarray*}
	where $t=\tan(\lambda).$
	
	Now for the case when $x_0<0,$ suppose that the coordinates of the point $M''$ are $(c,d)$. Equation \eqref{s3e10} still hold for this case, that is,
	\begin{equation}
	x_0^2=d^2x_0+c^2x_0-c(x_0^2+y_0^2-1).\label{s3e12}
	\end{equation}
	Now, consider right angle triangle $OQQ'$ in Figure \ref{fig3.5}.
	Since $|OQ|=1$, then $|OQ'|=\cos(\lambda)$ and $|QQ'|=\sin(\lambda).$ Hence the coordinates of the point $Q$ are $(-\cos(\lambda),\sin(\lambda)).$ Therefore, the slope of $OQ$ is $-\dfrac{\sin(\lambda)}{\cos(\lambda)}=-\tan(\lambda)$. Also, the slope of  $OM''$ is $\dfrac{d}{c}$. But then, slope of $OQ$ is equal to slope of $OM''$. Thus
	\begin{equation}
	-\tan(\lambda)=\frac{d}{c}\implies d+c\tan(\lambda)=0.\label{s3e13}
	\end{equation}
	Solving \eqref{s3e12} and \eqref{s3e13} simultaneously, we have that
	\begin{eqnarray*}
	c&=&\frac{x_0^2+y_0^2-\sqrt{x_0^4+y_0^4+2x_0^2(2t^2+1)+2y_0^2(x_0^2-1)+1}-1}{2x_0(t^2+1)},\text{ and }\\
	d&=&-\frac{tx_0^2+ty_0^2-t\left(\sqrt{4t^2x_0^2+x_0^4+2x_0^2y_0^2+y_0^4+2x_0^2-2y_0^2+1}-1\right)}{2x_0(t^2+1)},
	\end{eqnarray*}
	where $t=\tan(\lambda).$
	Therefore, we define the map $r:Y\to A$ by 
	\begin{equation}
	r(x_0,y_0)=
	\begin{cases}
	(a,b),\text{ if }x_0>0\\
	(c,d),\text{ if }x_0<0\\
	(0,0), \text{ if }x_0=0.
	\end{cases}\label{s3f1}
	\end{equation}
	where $a,b,c,d$ are as above. By construction, $r$ is continuous. Now we prove that $r$ is a deformation retract of $Y$ onto $A$.
	Let $h$ be the distance from the origin to any point $T(p_0,q_0)$ on $OP$ or $OQ$. Then the coordinate of $T$ is 
	\begin{equation}
	(p_0,q_0)=
	\begin{cases}
	 (h\cos(\lambda),h\sin(\lambda)),\text{ if }p_0>0\\
	 (-h\cos(\lambda),h\sin(\lambda)),\text{ if }p_0<0\\
	 (0,0),\text{ if }p_0=0.
	\end{cases}\label{s3e14}
	\end{equation}
	 Hence, every element of $A=OP\cup OQ$ has the form in \eqref{s3e14} above.  So to show that $r$ is indeed a
	 \newpage
	   retraction, we suppose that $(x_0,y_0)=(h\cos(\lambda),h\sin(\lambda))\in OP\subset A$ is arbitrary. This is for the case when $x_0>0$. So by substituting $x_0=h\cos(\lambda)$ and $y_0=h\sin(\lambda)$ into \eqref{s3f1}, we have that $$r(x_0,y_0)= (x_0,y_0).$$ 
	   Since $(x_0,y_0)$ is arbitrary, then $r(x_0,y_0)=(x_0,y_0),~\forall (x_0,y_0)\in OP$. Similarly, one can show for the case when $x_0<0,$ that is, when $(x_0,y_0)\in OQ$. For the case when $x_0=0,$ equation \eqref{s3e14} says that $(x_0,y_0)=(0,0)$, and thus $r(x_0,y_0)=(0,0)$ from \eqref{s3f1}. Therefore, $r$ is indeed a retraction.
	   
	   Now to conclude that $r$ is a deformation retract, we will show that $i\circ r\sim \mathrm{Id}_Y$, where $i:A\to Y$ is the inclusion map. The homotopy $F:Y\times [0,1]\to Y$ follows the circle when $x_0\ne 0$. But if $x_0=0,$ $F$ is a linear homotopy along the y-axis from $(0,y_0)$ to $(0,0)$. Hence $r$ is a deformation retract.
	\end{proof}	 
	\begin{pro} Let $X_m$ and $X_n$ be as defined in Proposition \ref{c3p313}. Then the spaces $(S^1\times D^2)\setminus K$ and $(D^2\times S^1)\setminus K$, where $K$ is a torus knot, deformation retract onto $X_m$ and $X_n$ respectively.
	\end{pro}
	\begin{proof}
		This follows immediately from Lemma \ref{c3l315}.
	\end{proof}  
	   
	   These two deformation retractions do not agree on their common domain, $(S^1 \times S^1 )\setminus K$.
	   This can be corrected by changing flows in the two solid tori so that $(S^1 \times  S^1 ) \setminus K$, both flows
	   are orthogonal to $K$. After this adjustment, we then have that $S^3 \setminus K$ deformation retracts onto $X=X_{m,n}.$  

\subsection{Calculating the Fundamental Group of $\boldsymbol{X=X_{m,n}}$} Now, we will use Van Kampen's theorem to compute $\pi_1(X)$. 
\begin{pro}
	$\pi_1(X)\cong \langle a,b | a^m=b^n\rangle$.
\end{pro}
\begin{proof}
To prove this, we suppose that $U$ is an open neighborhood of $X_m$ and $V$ is an open neighborhood of $X_n$ such that
\begin{itemize}
	\item[i)] $U$ and $V$ deformation retract onto $X_m$ and $X_n$ respectively,
	\item[ii)] $U\cup V=X_{m,n}$.
\end{itemize}
The spaces $U$ and $V$ are path connected since $X_m$ and $X_n$ are path connected because they are both homeomorphic to a cylinder as shown earlier. 
Since $U$ and $V$ are open neighborhoods of $X_m$ and $X_n$ respectively, then $U\cap V$ is an open neighborhood of $X_m\cap X_n$. Also, for the fact that $U$ and $V$ deformation retracts onto $X_m$ and $X_n$ respectively, then $U\cap V$ deformation retract onto $X_m\cap X_n$.

The space $U\cap V$ is path connected since $X_m\cap X_n$ is path connected because $X_m\cap X_n$ is a circle.
Let $x_0\in U\cap V$ be the basepoint, then van Kampen's theorem says that
\begin{equation}
\pi_1(X_{m,n})\cong \frac{\pi_1(U)\star \pi_1(V)}{N},\label{s3e16}
\end{equation}
where $N$ is a normal subgroup generated by $i_1(g)i_2(g)^{-1},~g\in\pi_1{(U\cap V)}$. Since $U$ and $V$ deformation retract onto $X_m$ and $X_n$ respectively, then by Theorem \ref{c2t257}, we have that $\pi_1(U)\cong \pi_1(X_m)$ and $\pi_1(V)\cong \pi_1(X_n)$. But then, $X_m$ is the mapping cylinder of the map $\alpha_m:S^1\to S^1$ and also $X_n$ is the mapping cylinder of the map $\alpha_n:S^1\to S^1$. Hence by Proposition \ref{c2pm}, both $X_m$ and $X_n$ deformation retract onto $S^1$. Again, by Theorem \ref{c2t257}, we have that $\pi_1(X_m)\cong \pi_1(S^1)$ and $\pi_1(X_n)\cong \pi_1(S^1)$. Therefore
\begin{eqnarray*}
\pi_1(U)&\cong&\pi_1(X_m)\cong \pi_1(S^1)\cong \mathbb{Z},\\
\pi_1(V)&\cong&\pi_1(X_n)\cong \pi_1(S^1)\cong \mathbb{Z}.
\end{eqnarray*}
Also, since $U\cap V$ deformation retracts onto $X_m\cap X_n$ and for the fact that $X_m\cap X_n=S^1\times\{\frac{1}{2}\}$. Then
$$\pi_1(U\cap V)\cong \pi_1(X_m\cap X_n)=\pi_1(S^1\times \{1/2\})\cong \mathbb{Z}.$$
Let $\alpha$ be a loop in $U\cap V$ which represents a generator of $\pi_1(U\cap V)$, then $\alpha$ is homotopic to a loop in $U$ representing $m$ times a generator, and also homotopic to a loop in $V$ representing $n$ times a generator.
Assuming $\pi_1(U)$ is generated by $a$, and $\pi_1(V)$ is generated by $b$, then $N$ is generated by $\alpha = i_1(g)i_2(g)^{-1}=a^mb^{-n}$. Hence \eqref{s3e16} implies that 
\begin{equation}
\pi_1(X_{m,n})\cong \dfrac{\langle a\rangle \star \langle b\rangle}{\langle a^mb^{-n}\rangle}=\langle{a,b|a^mb^{-n}}\rangle=\langle{a,b|a^m=b^n}\rangle.\label{s3e17}
\end{equation}
\end{proof}
\begin{thm}
	$\pi_1(\mathbb{R}^3\setminus K)\cong \langle a,b | a^m=b^n\rangle.$
\end{thm}
\begin{proof}
We have proved that $S^3\setminus K$ deformation retracts onto $X_{m,n}$. Hence by Theorem \ref{c2t257}, we have that
\begin{equation}
\pi_1(S^3\setminus K)\cong \pi_1(X_{m,n}).\label{s3e18}
\end{equation}
So by substituting \eqref{s3e17} into \eqref{s3e18}, we have that
\begin{equation}
\pi_1(S^3\setminus K)\cong \langle a,b|a^m=b^n\rangle.\label{s3e19}
\end{equation}
From \eqref{s3e4}, we have that $\pi_1(S^3\setminus K)\cong \pi_1(\mathbb{R}^3\setminus K)$. This implies that $\pi_1(\mathbb{R}^3\setminus K)\cong \pi_1(S^3\setminus K)$ because the relation $"\cong"$ is an equivalence relation. Hence 
\begin{equation}
\pi_1(\mathbb{R}^3\setminus K)\cong \langle a,b|a^m=b^n\rangle. \label{s3e20}
\end{equation}
Hence, we conclude that the knot group of a torus knot $K=K_{m,n}$ is $\pi_1(\mathbb{R}^3\setminus K)\cong \langle a,b|a^m=b^n\rangle.$
\end{proof}
\section{Description of the Structure of $\pi_1(\mathbb{R}^3\setminus K)$}
\vspace{-5pt}
In this section, we will describe the structure of the group $G_{m,n}:=\pi_1(\mathbb{R}^3\setminus K)\cong\langle a,b|a^m=b^n\rangle$. We will first show that the group $G_{m,n}$ is infinite cyclic when $m$ or $n$ is 1. Then we will show that the quotient group $G_{m,n}/{C}\cong\mathbb{Z}_m\star\mathbb{Z}_n$ with the assumption that $m,n>1$, where $C$ is the cyclic normal subgroup of $G_{m,n}$ generated by the element $a^m=b^n$. After this, we will show that $C=\langle a^m\rangle=\langle b^n\rangle$ is exactly the center of $G_{m,n}$. Lastly, we will show that $m$ and $n$ are uniquely determined by $G_{m,n}$.

\begin{pro}
	The group ${G_{m,n}}:=\pi_1(\mathbb{R}^3\setminus K)\cong\langle a, b| a^m=b^n\rangle$ is infinite cyclic when ${m}$ or ${n}$ is 1
\end{pro}
\begin{proof}
	When $m=1$, we have that 
	\begin{eqnarray*}
		G_{m,n}\cong\langle{a,b|a=b^n}\rangle
		=\langle {b^n,b}\rangle
		=\langle {b}\rangle
		=\mathbb{Z}.
	\end{eqnarray*}
The group $\langle {b^n,b}\rangle=\langle {b}\rangle$ since $b^n\in \langle b\rangle$. Also, when $n=1$, we have that
\begin{eqnarray*}
	G_{m,n}\cong\langle{a,b|a^m=b}\rangle
	=\langle {a,a^m}\rangle
	=\langle {a}\rangle
	=\mathbb{Z}.
\end{eqnarray*}
Since $\mathbb{Z}$ is infinite cyclic, then we conclude that whenever $m$ or $n$ is 1, $G_{m,n}$ is infinite cyclic.
\end{proof}
\newpage
\begin{pro}
	The quotient group ${G_{m,n}/C\cong \mathbb{Z}_m\star\mathbb{Z}_n}$, where ${C=\langle a^m\rangle=\langle b^n\rangle}$, and ${m,n>1}.$ 
\end{pro}
\begin{proof}
	First of all, we will show that $C\triangleleft G_{m,n}$, then we will show that $G_{m,n}/C\cong \mathbb{Z}_m\star\mathbb{Z}_n$. 
	
	Since $C=\langle{a^m}\rangle=\langle{b^n}\rangle$, then $C\subset G_{m,n}.$
	Also $C\ne\emptyset$ since $a^m\in C$. Now let $x,y\in C$, then there exists $p,q\in \mathbb{Z}$ such that $x=a^{mp}=b^{np}$, and $y=a^{mq}=b^{nq}.$
	Thus
	\begin{eqnarray*}
		xy^{-1}=a^{mp}a^{-mq}
		=a^{m(p-q)}
		\in C ~~\text{ (since $p-q\in\mathbb{Z}$)}.
	\end{eqnarray*}
	Hence, by Proposition \ref{s2p1}, $C\le G_{m,n}$. Now, let us show that the element $a^m=b^n$ commutes with all the elements of $G_{m,n}$ before we show that $C$ is a normal subgroup of $G_{m,n}.$
	
	Since $a^m=b^n$, then $a^{m}a=aa^{m}$, and $a^{m}b=b^{n}b=bb^{n}=ba^m$. Hence $a^m=b^n$ commutes with $a$ and $b$. But then, $a$ and $b$ are generators of $G_{m,n}$, thus for any element $h\in G_{m,n}$, one has
	\begin{equation}
	a^mh=ha^m, \text{ and } b^nh=hb^n,\text{ for all } h\in G_{m,n}.\label{s32e1}
	\end{equation}
	Now, let $c\in C$ be arbitrary, then there exists $t\in\mathbb{Z}$ such that $c=a^{mt}=b^{nt}$. Also, let $g\in G_{m,n}$ be arbitrary, then
	\begin{eqnarray*}
		gcg^{-1}&=&ga^{mt}g^{-1}
		=a^{mt}gg^{-1}~~\text{(from \eqref{s32e1})}
		=a^{mt}
		=c\in C.
	\end{eqnarray*}
	Hence $C$ is a normal subgroup of $G_{m,n}$. Now, since $G_{m,n}\cong\langle{a,b|a^m=b^n}\rangle$ and $C=\langle{a^m}\rangle=\langle{b^n}\rangle$, then
	\begin{eqnarray*}
		\frac{G_{m,n}}{C}&\cong&\frac{\langle{a,b|a^m=b^n}\rangle}{\langle{a^m}\rangle}\\
		&=&\langle{a,b|a^m=b^n, a^m}\rangle\\
		&=&\langle{a,b|a^m=b^n, a^m=1}\rangle\\
		&=&\langle{a,b|a^m=b^n=1}\rangle\\
		&=&\langle{a|a^m=1}\rangle\star \langle{b|b^n=1}\rangle\\
		&=&\langle{a|a^m}\rangle\star\langle{b|b^n}\rangle\\
		&=&\mathbb{Z}_m\star \mathbb{Z}_n.
	\end{eqnarray*}
\end{proof}
\begin{pro}
	The group ${C=\langle a^m\rangle=\langle b^n\rangle}$ is exactly the center of ${G_{m,n}}.$
\end{pro} 
\begin{proof}
	We want to show that $C=Z(G_{m,n})$, where $Z(G_{m,n})$ is the center of the group $G_{m,n}$. To do this, let $c\in C$ be arbitrary, then there exists $t\in \mathbb{Z}$ such that $c=a^{mt}=b^{nt}$. Also, let $h\in G_{m,n}$ be arbitrary, then $$ch=a^{mt}h=ha^{mt}~~\text{(from \eqref{s32e1})}=hc.$$
	Hence $c\in Z(G_{m,n})$. Therefore
	\begin{equation}
	C\subseteq Z(G_{m,n}).\label{s32e2}
	\end{equation}
	Next, we will show that $\mathbb{Z}_m\star\mathbb{Z}_n$ has a trivial center. 
	Since $|\mathbb{Z}_m|=m>1$, and $|\mathbb{Z}_n|=n>1$, then there exist $x\in\mathbb{Z}_m$ and $y\in\mathbb{Z}_n$ such that $x\ne e_{\mathbb{Z}_m},\text{ and } y\ne e_{\mathbb{Z}_n}.$
	Let $w\in \mathbb{Z}_m\star\mathbb{Z}_n$ be any nonempty reduced word. Without loss of generality, we suppose that $$w=w_1\dots w_kx_r,$$
	where $x_r$ is a nonempty reduced word of $\mathbb{Z}_m.$
	Let $y\in\mathbb{Z}_n$ such that $y\ne e_{\mathbb{Z}_n}$, then $y\in \mathbb{Z}_m\star \mathbb{Z}_n.$ Hence $yw=yw_1\dots w_kx_r$ must end with $x_r$ after reduction. But $wy=w_1\dots w_kx_ry$ does not end with $x_r$
	\newpage 
	since $x_r$ and $y$ cannot reduce with each other because they come from different group. 
	Therefore, $$yw\ne wy\implies w\notin Z(\mathbb{Z}_m\star \mathbb{Z}_n).$$
	\item Since $w$ is arbitrary reduced word of $\mathbb{Z}_m\star\mathbb{Z}_n$, it follows that $e$, the empty word, is the only element of $Z(\mathbb{Z}_m\star\mathbb{Z}_n)$. Thus $$Z(\mathbb{Z}_m\star \mathbb{Z}_n)=\{e\}.$$
	Now, since $\dfrac{G_{m,n}}{C}\cong \mathbb{Z}_m\star\mathbb{Z}_n$ as shown earlier, then 
	$Z\left(\frac{G_{m,n}}{C}\right)\cong Z(\mathbb{Z}_m\star\mathbb{Z}_n)=\{e\}$. But then, the 
	
	identity element of the quotient group $G_{m,n}/C$ is $C$, hence we have that
	\begin{equation}
	Z\left(G_{m,n}/C\right)\cong\{C\}.\label{s32e3}
	\end{equation}
	Now let us show that $Z(G_{m,n})\subseteq C.$ To do this, we suppose that $z\in Z(G_{m,n})$ is arbitrary. This implies that $ z\in G_{m,n}$, and thus  $zC\in G_{m,n}/C.$
	Now, let $z'C\in G_{m,n}/C$, then
	\begin{eqnarray*}
		zCz'C&=&Czz'C~~(\text{since $C\subseteq Z(G_{m,n})$})\\
		&=&Cz'zC~~(\text{since $z\in Z(G_{m,n})$})\\
		&=&z'CzC.
	\end{eqnarray*}
	Hence $zC\in Z(G_{m,n}/C)$.
	Thus \eqref{s32e3} implies that $zC=C.$ Therefore by Proposition \ref{s2p17}, $z\in C$. Since $z$ is arbitrary, then we have that 
	\begin{equation}
	Z(G_{m,n})\subseteq C.\label{s32e4}
	\end{equation}
	Hence \eqref{s32e2} and \eqref{s32e4} then imply that $Z(G_{m,n})=C.$
\end{proof}

\begin{thm}
	The integers ${m,n>1}$, are uniquely determined by ${G_{m,n}}$.
\end{thm}
\begin{proof}
	To prove this, we first show that the abelianization of $\mathbb{Z}_m\star\mathbb{Z}_n$ is $\mathbb{Z}_m\times \mathbb{Z}_n$. By definition of abelianization, one has that the abelianization of $\mathbb{Z}_m\star\mathbb{Z}_n$, denoted by  $\mathrm{Ab}(\mathbb{Z}_m\star\mathbb{Z}_n)$, is $$\mathrm{Ab}(\mathbb{Z}_m\star\mathbb{Z}_n)=\frac{\mathbb{Z}_m\star\mathbb{Z}_n}{\left[\mathbb{Z}_m\star\mathbb{Z}_n, \mathbb{Z}_m\star\mathbb{Z}_n\right]},$$
	where $\left[\mathbb{Z}_m\star\mathbb{Z}_n, \mathbb{Z}_m\star\mathbb{Z}_n\right]$ is the commutator subgroup of $\mathbb{Z}_m\star\mathbb{Z}_n.$ Therefore, we have that
	\begin{eqnarray*}
		\mathrm{Ab}(\mathbb{Z}_m\star\mathbb{Z}_m)&=&\frac{\mathbb{Z}_m\star\mathbb{Z}_n}{\left[\mathbb{Z}_m\star\mathbb{Z}_n, \mathbb{Z}_m\star\mathbb{Z}_n\right]}\\&=&\frac{\langle{a,b|a^m=1, b^n=1,}\rangle}{\langle{aba^{-1}b^{-1}}\rangle}\\
		&=&\langle{a,b|a^m=1, b^n=1, aba^{-1}b^{-1}=1}\rangle\\
		&=&\langle{a,b|a^m=1, b^n=1, ab=ba}\rangle\\
		&=&\mathbb{Z}_m\times\mathbb{Z}_n.
	\end{eqnarray*}
	Since $|\mathbb{Z}_m\times\mathbb{Z}_n|=mn$, then the product $mn$ is uniquely determined by $\mathbb{Z}_m\star\mathbb{Z}_n$, and therefore $mn$ is uniquely determined by $G_{m,n}.$ By taking $A=\mathbb{Z}_m$ and $B=\mathbb{Z}_n$ in Proposition \ref{s23p3}, we have that all torsion elements of $\mathbb{Z}_m\star\mathbb{Z}_n$ are conjugate of torsion elements of $\mathbb{Z}_m$ or $\mathbb{Z}_n$. Hence, if $x\in \mathrm{Tor}(\mathbb{Z}_m\star\mathbb{Z}_n)$, and that $\mathrm{ord}(x)=p\in\mathbb{Z}$, then there exist $y\in \mathbb{Z}_m\star \mathbb{Z}_n$, and $z\in \mathrm{Tor}(\mathbb{Z}_m)$ or $\mathrm{Tor}(\mathbb{Z}_n)$ such that
	$$x=yzy^{-1},\text{ and }x^p=e.$$
	But then $x^p=(yzy^{-1})^p=yz^py^{-1}=e\implies z^p=e$. Since $z\in\mathrm{Tor}(\mathbb{Z}_m)$ or $\mathrm{Tor}(\mathbb{Z}_n)$, then by Corollary \ref{s2c1}, $p$ divides $m$ or $n$. 
	This then implies that the maximum order of torsion elements of $\mathbb{Z}_m\star\mathbb{Z}_n$ is 
	\newpage
	the larger of $m$ or $n$.
	The maximum of $m$ or $n$, say $m$ is uniquely determined by $\mathbb{Z}_m\star\mathbb{Z}_n$. Hence $m$ is uniquely determined by $G_{m,n}.$
	Since $mn$ and $m$ are uniquely determined by $G_{m,n}$, then $n$ is uniquely determined by $G_{m,n}$.
\end{proof}
In the next section, we will give an algorithm for computing the knot group of an arbitrary knot. This is the generalization of this work.

\section{Calculating the Fundamental Group of Arbitrary Knots}
\label{c3s3.3}
The goal of this section is to compute $\pi_1(\mathbb{R}^3\setminus L)$, where $L$ is an arbitrary knot. In order to achieve this, we will use an algorithm for computing a presentation of $\pi_1(\mathbb{R}^3\setminus L)$ called the Wirtinger presentation.

\subsection{Wirtinger presentation\cite{allen}} Let $L\subset \mathbb{R}^3$ be an arbitrary knot. We position the knot to lie almost flat on a table so that $L$ consists of finitely many disjoint arcs $\alpha_i$, where it intersects the top of the table together with finitely many disjoint arcs $\beta_l$ as shown in Figure \ref{fig3.6}. Next, we construct a two dimensional complex $X$ as follows
\begin{figure}[htbp!]
	\begin{center}
		\begin{minipage}[b]{0.3\linewidth}
			\centering
			\includegraphics[width=1.\linewidth]{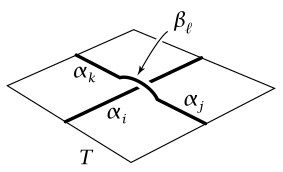} 
			\caption{Knot on a table \cite{allen}}\label{fig3.6} 
			\vspace{0.8ex}
		\end{minipage}\hspace{15pt}%%
		\begin{minipage}[b]{0.3\linewidth}
			\centering
			\includegraphics[width=1.\linewidth]{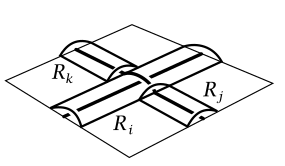} 
			\caption{Attaching $R_is$ \cite{allen}}\label{fig3.7}
			\vspace{0.8ex}
		\end{minipage}\hspace{15pt}
		\begin{minipage}[b]{0.3\linewidth}
			\centering
			\includegraphics[width=1.\linewidth]{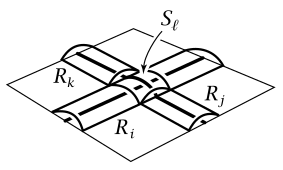} 
			\caption{Attaching $S_l$ \cite{allen}}\label{fig3.8}
			\vspace{0.8ex}  
		\end{minipage}
	\end{center}   
\end{figure}
\vspace{-15pt}

First, place a long, thin rectangular strip $R_i$ above each of the arcs $\alpha_i$, curved to run parallel to $\alpha_i$ along the full length of $\alpha_i$ and arched so that the two long edges of $R_i$ are identified with points of $T$. This is illustrated in Figure \ref{fig3.7}.

Any of the arcs $\beta_l$ that crosses over $\alpha_i$ are positioned to lie on $R_i.$ Finally, place a square $S_l$ above each arc $\beta_l.$ The square $S_l$ is bent downward along its four edges so that these edges are identified with points of the strips $R_i, R_j$, and $R_k.$ The two opposite edges of $S_l$ are identified with short edges of $R_j$ and $R_k$ while the other two opposite edges of $S_l$ are identified with two arcs crossing the interior of $R_i$ as shown in Figure \ref{fig3.8}.

With this construction, we see that $L$ is a subspace of $X$. But after we lift up $L$ slightly into the complement of $X$, then $\mathbb{R}^3\setminus L$ deformation retracts onto $X$. 
\begin{thm}
	The fundamental group of $\mathbb{R}^3\setminus L$ has a presentation with one generator $a_i$ for each strip $R_i$ and one relation of the form $a_ia_j=a_ka_i$ for each square $S_l$, where the indices are as in the figures above.\label{c3t332}
\end{thm}
\begin{proof}
 Since $\mathbb{R}^3\setminus L$ deformation retracts onto $X$, then by Theorem \ref{c2t257}, we have that
\begin{equation}
\pi_1(\mathbb{R}^3\setminus L)\cong \pi_1(X).\label{s33e1}
\end{equation}
To find $\pi_1(X)$,
we first start with a scaffolding space $Y$ as shown in Figure \ref{fig3.9} below. The knot under consideration is represented with the line segment $|PQ|$ and $|RS|$ in Figure \ref{fig3.9}. The segment $|PQ|$ is the overpass crossing and segment $|RS|$ is the underpass crossing. Space $Y$ is the shaded area which
\begin{figure}[htbp!]
	\begin{center}
		\begin{minipage}[b]{0.55\linewidth}
			\centering
			\includegraphics[width=0.9\linewidth]{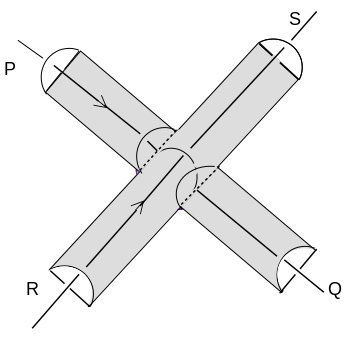} 
			\caption{\small{Space $Y$ (Shaded region with attached six arcs)}}\label{fig3.9} 
		\end{minipage}%%
		\begin{minipage}[b]{0.5\linewidth}
			\centering
			\includegraphics[width=0.95\linewidth]{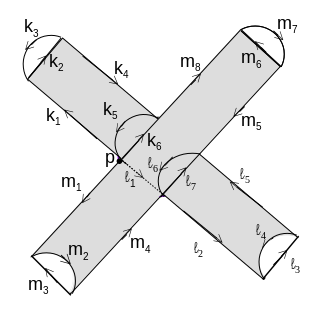} 
			\caption{\small{Loops around space $Y$}}\label{fig3.10}
		\end{minipage}  
		\begin{minipage}[b]{0.5\linewidth}
			\centering
			\includegraphics[width=0.95\linewidth]{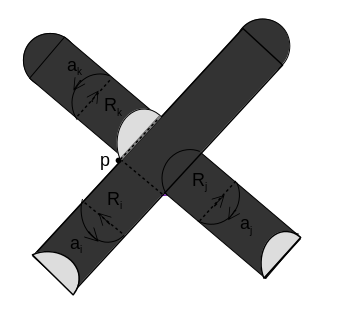} 
			\caption{Space $Z$.}\label{fig3.11} 
		\end{minipage}%%
		\begin{minipage}[b]{0.5\linewidth}
			\centering
			\includegraphics[width=0.95\linewidth]{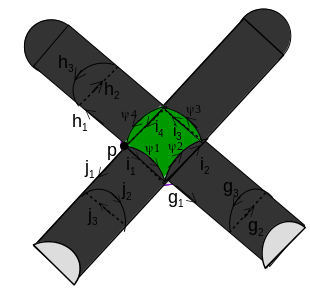} 
			\caption{Space $X$.}\label{fig3.12}
		\end{minipage}
	\end{center}   
\end{figure}
lies below $|PQ|$ and $|RS|$, together with the six attached arcs which pass over $|PQ|$ and $|RS|.$ It can be clearly seen in Figure \ref{fig3.10} that $$\pi_1(Y,p)\cong\langle a_k,a_j,a_i,a_k',a_j',a_i'\rangle,$$
where $a_k=k_1k_2k_3\bar{k_1},~ a_k'=k_6k_5,~a_j=l_1l_2l_3l_4\bar{l_2}\bar{l_1},~a_j'=l_1l_7l_6\bar{l_1},~a_i=m_1m_2m_3\bar{m_1},$ and $a_i'=k_6m_8m_7m_6\bar{m_8}\bar{k_6}.$ 

Next, we will attach 2-cells $R_i, R_j,$ and $R_k$ to space $Y$ get a new space $Z.$ The 2-cell $R_i$ is attached through the homomorphism $\phi_i:S^1\rightarrow Y$, where $\phi_i=m_1m_2m_4l_7\bar{m_5}\bar{m_7}\bar{m_8}\bar{k_6}.$ We must note that $R_i$ is
\newpage
kept under the overpass of segment $|PQ|$. The 2-cell $R_k$ is attached via the homomorphism $\phi_k:S^1\rightarrow Y$, where $\phi_k=\bar{k_5}\bar{k_4}k_3\bar{k_1}.$
Lastly, we attach $R_j$ via the homomorphism $\phi_j:S^1\rightarrow Y$, where $\phi_j=l_2\bar{l_4}l_5l_6.$
The new space formed, that is, space $Z$, is shown in Figure \ref{fig3.11} above.

To find the fundamental group of the space $Z$, we will make use of the Proposition \ref{s3p1} below. Consider the description below.

Suppose we attach a collection of 2-cells $e_\alpha^2$ to a path-connected space $X$ via maps $\psi_\alpha:S^1\to X$, producing a space $Y$. If $s_0$ is a basepoint of $S^1$, then $\psi_\alpha$ determines a loop at $\psi_\alpha(s_0)$ that we shall call $\psi_\alpha,$ even though technically loops are maps $I\to X$ rather than $S^1\to X$. For different $\alpha$'s, the basepoints $\psi_\alpha(s_0)$ of these loops $\psi_\alpha$ may not all coincide. 
To remedy this, choose a basepoint $x_0\in X$ and a path $\gamma_\alpha$ in $X$ from $x_0$ to $\psi_\alpha(s_0)$ for each $\alpha$. Then $\gamma_\alpha\psi_\alpha\bar{\gamma_\alpha}$ is a loop at $x_0$. This loop may not be nullhomotopic in $X$, but it will certainly be nullhomotopic after the cell $e_\alpha^2$ is attached. Thus the normal subgroup $N\subset \pi_1(X,x_0)$ generated by all the loops $\gamma_\alpha\psi_\alpha\bar{\gamma_\alpha}$ for varying $\alpha$ lies in the kernel of the map $\pi_1(X,x_0)\to \pi_1(Y,x_0)$ induced by the inclusion $X\hookrightarrow Y$.

\begin{pro}\cite{allen}:
	 If the space $Y$ is obtained from $X$ by attaching 2-cells as described above, then the inclusion $X\hookrightarrow Y$ induces a surjection $\pi_1(X,x_0)\to\pi_1(Y,x_0)$ whose kernel is $N$. Thus $$\pi_1(Y)\cong \pi_1(X)/N.\label{s3p1}$$
\end{pro}\vspace{-15pt}
\begin{proof}
	The proof of this proposition can be found in [\cite{allen}, page 50].
\end{proof}
So by applying this proposition to $Z$, we then have that 
\begin{equation}
\pi_1(Z,p)\cong \pi_1(Y,p)/N,\label{s33e2}
\end{equation}
where $N$ is a normal subgroup of $(Y,p)$ generated by $\{\phi_i, l_1\phi_j\bar{l_1},\phi_k\}$. But then if we suppose that
$$r_i=m_1\bar{m_3}m_4l_7\bar{m_5}m_6\bar{m_8}\bar{k_6},~r_j=l_1l_2l_3l_5\bar{l_7}\bar{l_1},\text{ and }r_k=k_6\bar{k_4}\bar{k_2}\bar{k_1},$$
Then
\begin{eqnarray*}
	a_ir_i\bar{a_i'}&=&m_1m_2m_3\bar{m_1}m_1\bar{m_3}m_4l_7\bar{m_5}m_6\bar{m_8}\bar{k_6}k_6m_8\bar{m_6}\bar{m_7}\bar{m_8}\bar{k_6}
	=m_1m_2m_4l_7\bar{m_5}\bar{m_7}\bar{m_8}\bar{k_6}
	=\phi_i,\\
	\bar{a_j}r_ja_j'&=&l_1l_2\bar{l_4}\bar{l_3}\bar{l_2}\bar{l_1}l_1l_2l_3l_5\bar{l_7}\bar{l_1}l_1l_7l_6\bar{l_1}
	=l_1l_2\bar{l_4}l_5l_6\bar{l_1}
	=l_1\phi_j\bar{l_1},\text{ and }\\
	\bar{a_k'}r_ka_k&=&\bar{k_5}\bar{k_6}k_6\bar{k_4}\bar{k_2}\bar{k_1}k_1k_2k_3\bar{k_1}
	=\bar{k_5}\bar{k_4}k_3\bar{k_1}
	=\phi_k.
\end{eqnarray*}
But $r_i, r_j,$ and $r_k$ are homotopic to a point. In fact, they are homotopic to the basepoint $p$. Hence
$$\phi_i=a_ir_i\bar{a_i'}\simeq a_i\bar{a_i'},~~l_1\phi_j\bar{l_1}=\bar{a_j}r_ja_j'\simeq \bar{a_j}a_j',\text{ and }\phi_k=\bar{a_k'}r_ka_k\simeq \bar{a_k'}a_k.$$
Thus,\vspace{-10pt}
\begin{eqnarray*}
	\pi_1(Z,p)&\cong &\pi_1(Y,p)/N\\
	&\cong &\dfrac{\langle a_k,a_j,a_i,a_k',a_j',a_i'\rangle}{\langle a_i\bar{a_i'}, \bar{a_j}a_j',\bar{a_k'}a_k\rangle}\\
	&=&{\langle a_k,a_j,a_i,a_k',a_j',a_i'|a_i\bar{a_i'}=1, \bar{a_j}a_j'=1, \bar{a_k'}a_k=1\rangle}\\
	&=&\langle a_k,a_j,a_i,a_k',a_j',a_i'|a_i=a_i', a_j'=a_j, a_k=a_k'\rangle\\
	&=&\langle a_i,a_j,a_k\rangle.
\end{eqnarray*}
Finally, we attach a 2-cell $S_l$, to get $X$, via the homomorphism $\psi:S^1\rightarrow Z$, where
$\psi = \psi_1\psi_2\psi_3\psi_4.$
The new space formed, that is, space $X$, is exactly the space we got from the Wirtinger algorithm as shown in Figure \ref{fig3.12}. We must ensure that $S_l$ lies on top of the overpass of segment $|PQ|$. Again, by Proposition \ref{s3p1}, we have that $$\pi_1(X,p)\cong\frac{\pi_1(Z,p)}{N},$$
where $N$ is a normal subgroup of $\pi_1(Z,p)$ generated by $\psi=\psi_1\psi_2\psi_3\psi_4.$ By comparing Figures \ref{fig3.11}  
to Figure \ref{fig3.12}, we have that 
$a_i=j_1j_2j_3\bar{j_1}, ~a_k=h_1h_2h_3\bar{h_1},\text{ and }a_j=i_1g_1g_2g_3\bar{g_1}\bar{i_1}.$
More so, we have that
	$\psi_1\bar{i_1}\simeq j_1j_2j_3\bar{j_1}$,
	$i_1\psi_2i_3i_4\simeq i_1\psi_2\bar{i_2}\bar{i_1}\simeq i_1(g_1\bar{g_3}\bar{g_2}\bar{g_1})\bar{i_1}$,
	$\bar{i_4}\bar{i_3}\psi_3i_4\simeq \bar{i_4}(i_4i_1\bar{\psi_1}\bar{i_4})i_4=i_1\bar{\psi_1}\simeq j_1\bar{j_3}\bar{j_2}\bar{j_1}$, and 
	$\bar{i_4}\psi_4\simeq h_1h_2h_3\bar{h_1}.$
Therefore,
\begin{eqnarray*}
	\psi&=&\psi_1\psi_2\psi_3\psi_4\\
	&\simeq& \psi_1(\bar{i_1}i_1)\psi_2(i_3i_4\bar{i_4}\bar{i_3})\psi_3(i_4\bar{i_4})\psi_4\\
	&\simeq&(\psi_1\bar{i_1})(i_1\psi_2 i_3i_4)(\bar{i_4}\bar{i_3}\psi_3i_4)(\bar{i_4}\psi_4)\\
	&\simeq&(j_1j_2j_3\bar{j_1})(i_1g_1\bar{g_3}\bar{g_2}\bar{g_1}\bar{i_1})(j_1\bar{j_3}\bar{j_2}\bar{j_1})(h_1h_2h_3\bar{h_1})\\
	&=&a_i\bar{a_j}\bar{a_i}a_k.
\end{eqnarray*}
Hence $N$ is generated by $\psi\simeq a_i\bar{a_j}\bar{a_i}a_k.$ Thus $N\simeq \langle a_i\bar{a_j}\bar{a_i}a_k\rangle$. But if a group is generated by an element, the group is also generated by the inverse of that element. The inverse of $a_i\bar{a_j}\bar{a_i}a_k$ is $\bar{a_k}a_ia_j\bar{a_i}$. Hence $N\simeq\langle \bar{a_k}a_ia_j\bar{a_i}\rangle.$
Therefore 
\begin{eqnarray}
	\pi_1(X,p)\cong\frac{\pi_1(Z,p)}{N}
	\cong\frac{\langle a_k,a_j,a_i\rangle}{\langle \bar{a_k}a_ia_j\bar{a_i}\rangle}
	=\langle a_i,a_j,a_k|\bar{a_k}a_ia_j\bar{a_i}\rangle
	=\langle a_i,a_j,a_k|a_ia_j=a_ka_i\rangle.\label{s33e3}
\end{eqnarray}
From \eqref{s33e1}, we have that 
\begin{equation}
\pi_1(\mathbb{R}^3\setminus L)\cong \pi_1(X)\cong \langle a_i,a_j,a_k|a_ia_j=a_ka_i\rangle.\label{c3e334}
\end{equation}
So, if the knot has $n$ arcs, we are going to have $n$ rectangular strips $R_i$, and $n$ squares $S_l$. For each of these rectangular strips, we conclude from \eqref{c3e334} that $\pi_1(\mathbb{R}^3\setminus L)$ has a presentation with one generator $a_i$, and one relation of the form $a_ia_j=a_ka_i$ for each square $S_l$.
\end{proof}
\begin{exa}
	Let us apply this result to compute the fundamental group of a trefoil knot, $K_{2,3}$. A trefoil knot has 3 arcs as shown in Figure \ref{fig314} below.  Hence, by Theorem \ref{c3t332}, we have that 
	\begin{eqnarray}
	\pi_1(\mathbb{R}^3\setminus K_{2,3})&\cong&\langle a_1, a_2, a_3|a_1a_2=a_3a_1, a_2a_3=a_1a_2, a_3a_1=a_2a_3\rangle\nonumber\\
	&=&\langle a_1,a_2,a_3|~a_1a_2a_1^{-1}=a_3, ~a_2a_3a_2^{-1}=a_1,~a_3a_1a_3^{-1}=a_2\rangle\nonumber\\
	&=&\langle a_1,a_2|~a_2a_1a_2a_1^{-1}a_2^{-1}=a_1,~a_1a_2a_1^{-1}a_1a_1a_2^{-1}a_1^{-1}=a_2\rangle\nonumber\\
	&=&\langle a_1,a_2|~a_2a_1a_2=a_1a_2a_1\rangle.\label{leq2}
	\end{eqnarray}
 Take $b=a_2a_1$ and $a=a_1b$, then $a_1=ab^{-1}$, and $a_2=ba_1^{-1}=bba^{-1}=b^2a^{-1}$. Substituting these into equation \eqref{leq2}, we then have that 
 \begin{eqnarray}
	 \pi_1(\mathbb{R}^3\setminus K_{2,3})&\cong&\langle a, b|~ba_2=a_1b~\rangle\nonumber\\
	 &=&\langle a, b|~bb^2a^{-1}=ab^{-1}b~\rangle\nonumber\\
	 &=&\langle a, b|~b^3a^{-1}=a~\rangle\nonumber\\
	 &=&\langle a, b|~b^3=a^2~\rangle.\label{leq3}
 \end{eqnarray}
 By substituting $m=2$ and $n=3$ in \eqref{s3e20}, we have that $\pi_1(\mathbb{R}^3\setminus K_{2,3})=\langle a,b|~ a^2=b^3\rangle$, which is exactly the same group in \eqref{leq3} above.
\end{exa}
\newpage
\begin{exa}
	We want to compute the fundamental group of a knot different from torus knot using our previous result. The knot is given in Figure \ref{fig313} below.
\end{exa}
\begin{figure}[htbp!]
	\begin{center}
		\hspace{-40pt}
		\begin{minipage}[b]{0.3\linewidth}
			\centering
			\includegraphics[width=1.\linewidth]{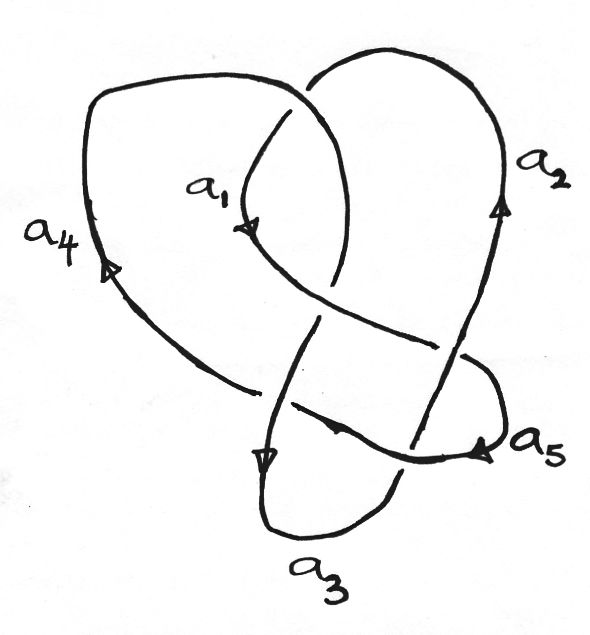} 
			\caption{A knot}\label{fig313} 
		\end{minipage}\hspace{70pt}%%
		\begin{minipage}[b]{0.3\linewidth}
			\centering
			\includegraphics[width=1.2\linewidth]{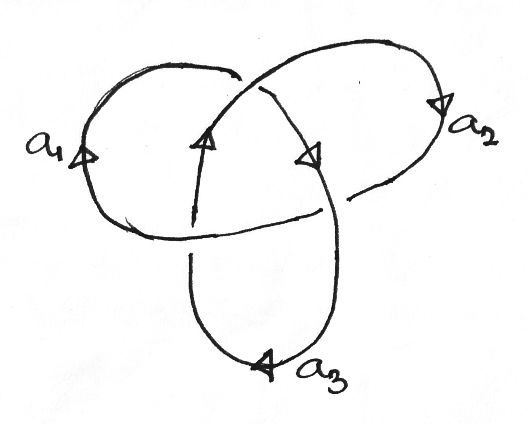} 
			\caption{Trefoil knot}\label{fig314}
		\end{minipage}%\hspace{45pt} 
	\end{center}  
\end{figure}
Let us denote this knot by $L$. Then $L$ has five crossings and five arcs, hence by Theorem \ref{c3t332} , we have that
\begin{eqnarray}
\pi_1(\mathbb{R}^3\setminus L)&=&\langle a_1,a_2,a_3,a_4,a_5| a_4a_1=a_2a_4, a_1a_3=a_4a_1, a_2a_5=a_1a_2,a_5a_2=a_3a_5,a_3a_4=a_5a_3\rangle\nonumber\\
&=&\langle a_1,a_2,a_3,a_4,a_5| a_1=\bar{a_4}a_2a_4, a_3=\bar{a_1}a_4a_1, a_5=\bar{a_2}a_1a_2, a_3=a_5a_2\bar{a_5}, a_4=\bar{a_3}a_5a_3\rangle\nonumber\\
&=&\langle a_1,a_2,a_3,a_4,a_5| a_1=\bar{a_4}a_2a_4, a_5=\bar{a_2}a_1a_2, a_3=a_5a_2\bar{a_5}, a_4=\bar{a_3}a_5a_3\rangle. \label{c3e337}
\end{eqnarray}
Since $a_3=a_5a_2\bar{a_5}$, then $a_4=\bar{a_3}a_5a_3=a_5\bar{a_2}\bar{a_5}a_5a_5a_2\bar{a_5}=a_5\bar{a_2}a_5a_2\bar{a_5}.$ Hence \eqref{c3e337} becomes 
\begin{equation}
\pi_1(\mathbb{R}^3\setminus L)= \langle a_1,a_2,a_4,a_5| a_1=\bar{a_4}a_2a_4, a_5=\bar{a_2}a_1a_2, a_4=a_5\bar{a_2}a_5a_2\bar{a_5}\rangle.\label{c3e338}
\end{equation}
Also, since $a_4=a_5\bar{a_2}a_5a_2\bar{a_5}$, then
$a_1=\bar{a_4}a_2a_4=a_5\bar{a_2}\bar{a_5}a_2\bar{a_5}a_2a_5\bar{a_2}a_5a_2\bar{a_5}.$Thus \eqref{c3e338} becomes
\begin{eqnarray}
	\pi_1(\mathbb{R}^3\setminus L)= \langle a_1,a_2,a_5| a_1=a_5\bar{a_2}\bar{a_5}a_2\bar{a_5}a_2a_5\bar{a_2}a_5a_2\bar{a_5}, ~a_5=\bar{a_2}a_1a_2\rangle.\label{c3e339}
\end{eqnarray}
Also, since since $a_1=\bar{a_4}a_2a_4=a_5\bar{a_2}\bar{a_5}a_2\bar{a_5}a_2a_5\bar{a_2}a_5a_2\bar{a_5}$, then 
$$ a_5=\bar{a_2}a_1a_2=\bar{a_2}\bar{a_4}a_2a_4a_2=\bar{a_2}a_5\bar{a_2}\bar{a_5}a_2\bar{a_5}a_2a_5\bar{a_2}a_5a_2\bar{a_5}a_2.$$
Therefore, \eqref{c3e339} becomes
\begin{eqnarray*}
\pi_1(\mathbb{R}^3\setminus L)&=& \langle a_2,a_5|a_5=\bar{a_2}a_5\bar{a_2}\bar{a_5}a_2\bar{a_5}a_2a_5\bar{a_2}a_5a_2\bar{a_5}a_2\rangle\\
&=&\langle a_2,a_5|a_2a_5\bar{a_2}a_5=a_5\bar{a_2}\bar{a_5}a_2\bar{a_5}a_2a_5\bar{a_2}a_5a_2\rangle.
\end{eqnarray*}
In the next chapter, we give the conclusion of this work by summarizing the results that were obtained. 

\chapter{Conclusion}
\label{chap4}
%The objective of this work, as stated in Section \ref{c1s1} of Chapter \ref{chap1}, was to compute and describe the fundamental group of $\mathbb{R}^3\setminus K$. We were able to achieve this in Chapter \ref{chap3}. We even extended our calculation to the computation of fundamental group of an arbitrary knot in Section \ref{c3s3.3} of Chapter \ref{chap3}. 
%
%In conclusion, we found that the fundamental group of a torus knot $K=K_{m,n}$ is
%\begin{equation}
%\pi_1(\mathbb{R}^3\setminus K)=\langle a,b|a^m=b^n\rangle.\label{conc1}
%\end{equation}
The main objective of this work was to compute and describe the fundamental group of $\mathbb{R}^3\setminus K$, where $K$ is a torus knot. To do this, we first showed that the fundamental group of $\mathbb{R}^3\setminus K$ is isomorphic to the fundamental group of $S^3\setminus K$. Then, we showed that $S^3\setminus K$ deformation retracts onto a 2-dimensional complex whose fundamental group is $\langle a,b|~a^m=b^n\rangle$. Hence,
$$\pi_1(\mathbb{R}^3\setminus K)\cong \langle a,b|~a^m=b^n\rangle.$$
Further analysis of $\pi_1(\mathbb{R}^3\setminus K)$ shows that the group is infinite cyclic whenever $m$ or $n$ is 1 and that the center of the group is $\langle a^m\rangle=\langle b^n\rangle$. Also, we succeeded in showing that the quotient group ${\pi_1(\mathbb{R}^3\setminus K)}/{\langle a^m\rangle}$ is isomorphic to the free product $\mathbb{Z}_m\star \mathbb{Z}_n$, and hence the integers $m$ and $n$ are uniquely determined by the group $\pi_1(\mathbb{R}^3\setminus K).$

%Furthermore, we found that for an arbitrary knot $L$ with $n$ arcs, the fundamental group of $\mathbb{R}^3\setminus L$ has a presentation with one generator $a_i$ for each arc, and one relation of the form $a_ia_j=a_ka_i$ for each crossing. 
Furthermore, we extended our calculation to the computation of the fundamental group of an arbitrary knot using the Wirtinger presentation. We found that for an arbitrary knot $L$ with $n$ arcs, the fundamental group of $\mathbb{R}^3\setminus L$ has a presentation with one generator $a_i$ for each arc, and one relation of the form $a_ia_j=a_ka_i$ for each crossing.

Finally, we computed the fundamental group of a trefoil knot using the result in the previous paragraph. We compared this fundamental group to the one in the first paragraph and discovered that they are the same.

 % Conclusion is usually a chapter itself. 
%\input{chapter5}
% This is where we stop counting pages.
% Acknowledgements and References (and appendices) are not counted.
%-----------------------------------------------------------------------------
% comment the next four lines if you do not have an appendix
% Appendices are usually not counted, and are only allowed if you
% speak to Prof. Marco about the need for it. Get approval first,
% and use this format. An appendix is not necessary for an
% essay's completeness.
\appendix
\renewcommand*{\theHchapter}{\Alph{chapter}}

\endappendix
%-----------------------------------------------------------------------------
% See the acknowledgement.tex file and follow the instructions there.
\chapter*{Acknowledgements}
\addcontentsline{toc}{chapter}{Acknowledgements}
% Don't change anything above this.
% Overly long acknowledgements are not professsional.
% You are free to use your native language 
 
Foremost, I would like to express my profound gratitude to God Almighty for the gift of life and for seeing me through this phase of my academic career. A special gratitude and deep regards go to my supervisors, Dr. Paul Arnaud Songhafouo Tsopm\'en\'e and Prof. Donald Stanley, for their exemplary monitoring, guidance, and enthusiasm throughout the course of this essay. Also, I am grateful to my tutor, Dr. Ihechukwu Chinyere, for all his corrections and words of encouragement during this work.

I would like to acknowledge with much appreciation the crucial role of the center president, Prof. Mama, the academic director, Prof. Marco, and the other staff members of AIMS Cameroon. Special thanks to the entire family of AIMS for making my stay here to be a success. I will not forget to thank my tutors for their support all through.

Words alone cannot express my gratitude to my parents, Mr. and Mrs. Mustapha, for their prayers, love and for leading me to the right path. Also, I thank my brother, Abass, for his advice and support.

To my classmates, thanks for your cooperation and for believing in Pan-Africanism. Finally, I thank my friends, Rabiat, Aisha, Tata Yolande, Mayar, Alaa, Romario, Adex, Seck, Tonton Yemi, Arnaud and mes autres amis for being there for me always and for all the fun we have had in the last ten months. You guys really made AIMS Cameroon worth staying for me. I love you all.
%-----------------------------------------------------------------------------
% Note the errata page is not for now, it is for use during the examination.
% Not that you're going to have any errata.
%-----------------------------------------------------------------------------
% THE BIBLIOGRAPHY 
% Bibliography styles define how the bibliography is listed and formatted.
%We use the "apa"  bibliography style here to show refernces in alphabetical order
% If your supervisor and tutor confirms that the field 
% you work in (physics, finance, etc.) generally use another
% style than plain, then you may change to that style. For example, You can use "apa" or "acm" that can arrange references alphabetically
\renewcommand{\bibname}{References}
% \nocite{*} %Uncommenting this line will allow all references in the bib file to appear in the bibliography (i.e. even the non-cited references).% This is not advisable and not allowed here in AIMS. 
\bibliographystyle{apa} 
\bibliography{ilyas}
\addcontentsline{toc}{chapter}{References}
%-----------------------------------------------------------------------------
\end{document}